\documentclass{imsart}

\RequirePackage[numbers]{natbib}

\usepackage{amsthm,amsmath,amssymb,bbm,graphicx,mathabx,mathrsfs,pdflscape}
\usepackage{tikz}
\usetikzlibrary{arrows,%
                petri,%
                topaths}%
\usepackage{tkz-berge}

\startlocaldefs

\theoremstyle{plain}

\newtheorem{theorem}{Theorem}[section]

\newtheorem{proposition}[theorem]{Proposition}
\newtheorem{corollary}[theorem]{Corollary}

\newtheorem{lemma}{Lemma}[section]

\theoremstyle{definition}
\newtheorem{definition}[theorem]{Definition}
\newtheorem{remark}{Remark}[section]
\newtheorem{example}[theorem]{Example}

\def\ci{\perp\!\!\!\perp}                        

\newcommand{\bfa} {\mbox{\boldmath $\alpha$}}

\newcommand{\bft} {\mbox{\boldmath $\theta$}}

\newcommand{\bfps} {\mbox{\boldmath $\psi$}}

\newcommand{\bfl} {\mbox{\boldmath $\lambda$}}

\newcommand{\bfy} {\mbox{\boldmath $y$}}
\newcommand{\bfi}{\mbox{\boldmath $i$}}

\newcommand{\bfn}{\mbox{\boldmath $n$}}

\newcommand{\bs}[1]{\boldsymbol{#1}}
\newcommand{\mc}[1]{\mathcal{#1}}
\def\m{\mathcal}

\newcommand{\ms}[1]{\mathscr{#1}}

\newcommand{\bb}[1]{\mathbb{#1}}

\newcommand{\ind}{\mathbbm{1}}
\newcommand{\tlam}{\tilde{\lambda}}
\newcommand{\rnkp}{\text{rnk}^+_P}

\newcommand{\rnkt}{\text{rnk}^+_T}
\newcommand{\rnkc}{\text{rnk}^+_{cT}}

\newcommand{\spt}{\text{spt}}
\newcommand{\teie}{\theta_E(\bfi_E)}

\newcommand{\intlevj}[1]{C_{\theta}^{(#1)}}
\newcommand{\indset}{H}
\newcommand{\indsetj}[1]{H_{#1}}

\newcommand{\barh}{\boldsymbol{\bar{H}}}

\def\T{{ \mathrm{\scriptscriptstyle T} }}

\endlocaldefs

\begin{document}
\begin{frontmatter}

\title{Tensor decompositions and sparse log-linear models}
\runtitle{Tensor rank of log-linear models}

\begin{aug}
\author{\fnms{James E.} \snm{Johndrow}\thanksref{t2}\ead[label=e1]{james.johndrow@duke.edu}},
\author{\fnms{Anirban} \snm{Bhattacharya}\thanksref{}\ead[label=e2]{anirbanb@stat.tamu.edu}}
\and
\author{\fnms{David B.} \snm{Dunson}\thanksref{t2} 
\ead[label=e3]{dunson@duke.edu}
\ead[label=u1,url]{http://www.isds.duke.edu/\textasciitilde dunson/}}

\thankstext{t2}{This research was partially support by grants ES017436 and ES017240 from the National Institute of Environmental Health Sciences (NIEHS) of the US National Institutes of Health.}
\runauthor{J.E. Johndrow et al.}

\affiliation{Duke University\thanksmark{m1} and Texas A\&M University\thanksmark{m2}}

\address{114 Old Chemistry Building\\
Duke University\\ 
Durham, NC 27708 \\
\printead{e1}\\
\printead{e3}\\
\printead{u1}} 

\address{ 155 Ireland Street \\
College Station, TX 77843 \\ 
\printead{e2}}
\end{aug}

\begin{abstract}
Contingency table analysis routinely relies on log linear models, with latent structure analysis providing a common alternative.  Latent structure models lead to a low rank tensor factorization of the probability mass function for multivariate categorical data, while log linear models achieve dimensionality reduction through sparsity.  Little is known about the relationship between these notions of dimensionality reduction in the two paradigms. We derive several results relating the support of a log-linear model to the nonnegative rank of the associated probability tensor. Motivated by these findings, we propose a new collapsed Tucker class of tensor decompositions, which bridge existing PARAFAC and Tucker decompositions, providing a more flexible framework for parsimoniously characterizing multivariate categorical data. Taking a Bayesian approach to inference, we illustrate advantages of the new decompositions in simulations and an application to functional disability data.
\end{abstract}



\end{frontmatter}


\section{Introduction}
Parsimonious models for contingency tables are of growing interest due to the routine collection of data on moderate to large numbers of categorical variables. We study the relationship between two paradigms for inference in contingency table models: the log-linear model (\cite{fienberg2007three}, \cite{bishop2007discrete}, \cite{agresti2002categorical}) and latent structure models (\cite{stouffer1950measurement}, \cite{gibson1955extension}, \cite{lazarsfeld1968latent}, \cite{anderson1954estimation}, \cite{madansky1960determinantal}, \cite{haberman1974log}, \cite{goodman1974exploratory}) that induce a tensor decomposition of the joint probability mass function (\cite{dunson2009nonparametric}, \cite{bhattacharya2012simplex}). In particular, we aim to understand situations where the joint probability corresponding to a sparse log-linear model has a low-rank tensor factorization, allowing us to connect the seemingly distinct notions of parsimony in the two parameterizations.  

Let $V = \{1, \ldots, p\}$ denote a set of $p$ categorical variables. We use $(y_j, j \in V)$ to denote variables, with $y_j \in \m I_j$ having $d_j = |\m I_j|$ levels. Without loss of generality, we assume $\m I_j = \{1, \ldots, d_j\}$. Let $\m I_V = \bigtimes_{j \in V} \m I_j$. Elements of $\m I_V$ are referred to as cells of the contingency table; there are $\prod_{j=1}^p d_j$ cells in total. We generically denote a cell by $\bfi$, with $\bfi = (i_1, \ldots, i_p) \in \m I_V$. The joint probability mass function of $\bfy = (y_1,\ldots,y_p)$ is denoted by $\pi$, with 
\begin{align}
\pi_{i_1\ldots i_p} = Pr(y_1 = i_1, \ldots y_p= i_p), \quad \bfi \in \m I_V. 
\end{align}
A $p$-way tensor $T \in \bb R^{d_1 \times \ldots \times d_p}$ is a multiway-array which generalizes matrices to higher dimensions \cite{kolda2009tensor}. Two common forms of tensor decomposition which extend the matrix singular value decomposition are the PARAFAC \cite{harshman1970foundations} and Tucker \cite{tucker1966some, de2000multilinear, de2000best} decompositions. Note that $\pi = (\pi_{i_1 \ldots i_p})_{\bfi \in \m I_V}$ 
can be identified with a $\bb R^{d_1 \times \ldots \times d_p}$-{\em probability tensor}, which is a non-negative tensor with entries summing to one. Given $n$ i.i.d. replicates of $\bfy$, let $\bfn(\bfi)$ denote the cell-count of cell $\bfi$. We assume the cell counts are multinomially distributed according to the probabilities in $\pi$. 

Inference for contingency tables often employs log-linear models that express the logarithms of the entries in $\pi$ as a linear function of parameters related to the index of each cell. Most of these parameters relate to interactions between the variables \cite{agresti2002categorical}. A saturated log linear model has as many parameters as $\pi$ has cells.  To reduce dimensionality, it is common to assume a large subset of the interaction parameters are zero, and estimate the model using $L_1$ regularization \cite{roth2008group, nardi2012log} or Bayesian model averaging \cite{dobra2011copula, dobra2004sparse, massam2009conjugate}. Zero interaction terms are easily interpreted in terms of conditional and marginal independence relationships among the variables. A significant literature exists on Bayesian inference for log-linear models, focusing mainly on the development of novel conjugate priors \cite{dawid1993hyper, massam2009conjugate}, model selection/averaging \cite{letac2012bayes}, and stochastic search algorithms to explore the model space (e.g. \cite{dobra2010mode}). 

An alternative approach is to assume that the $p$ variables are conditionally independent given one or more discrete latent class indices, with dependence induced upon marginalization over the latent variable(s). The attractiveness of such latent class models arises partly from easy model fitting using data-augmentation, with a Bayesian nonparametric formulation allowing the number of latent classes to be learnt from the data \cite{dunson2009nonparametric}. \cite{dunson2009nonparametric} showed that a single latent class model is equivalent to a reduced-rank non-negative PARAFAC decomposition of the joint probability tensor $\pi$, while the multiple latent class model in \cite{bhattacharya2012simplex} implied a Tucker decomposition. See also \cite{zhou2013bayesian} and \cite{kunihama2012bayesian} for extensions of these models to more complex settings. 


Latent class models and log-linear models can be unified within a larger class of graphical models with observed and unobserved variables (see e.g. \cite{lauritzen1996graphical, humphreys2003variational}). In particular, \cite{geiger2001stratified} describes relationships between the number of components in a PARAFAC expansion of $\pi$ and the topological structure of the corresponding parameter space of a log-linear model, with consequences for estimation and selection in latent structure models. Others have established additional connections between latent structure models and the algebraic topology of the log-linear model \cite{settimi1998geometry, smith2003bayesian, rusakov2002asymptotic, garcia2005algebraic, fienberg2007maximum, letac2012bayes}.

These two classes of models impose sparsity (or parsimony) in seemingly different ways, and to best of our knowledge, no connection has been established yet. The class of sparse log-linear models is often considered a desirable data generating class in high dimensional settings for flexibility and ease of interpretation, and it is important to determine whether there exist low-rank expansions for probability tensors corresponding to sparse log-linear models. Determining whether a nontrivial relationship exists is a major focus of the paper. Working with a class of weakly hierarchical log-linear models, we provide precise bounds on the tensor ranks of sparse log-linear models. There are limited results on ranks of higher-order tensors, and the techniques developed here may be of independent interest. 

A secondary goal of this work is to leverage insights from our theoretical study to propose more flexible latent structure models. Our results show that the effective number of parameters in latent structure models will grow exponentially in the number of variables under some circumstances. We propose a new tensor decomposition -- referred to as the collapsed Tucker factorization -- that generalizes PARAFAC and Tucker factorizations, and leads to more parsimonious representations of many probability tensors corresponding to sparse log-linear models.  We propose Bayes methodology for analyzing data under this new factorization.

This paper is organized as follows. Section 2 introduces notation and provides background relevant to log-linear models and latent structure models. Section 3 provides our main theoretical results on the rank of probability tensors corresponding to sparse log-linear models. Section 4 presents additional results that expand upon those in section 3 and motivate the proposed collapsed Tucker model. Section 5 presents a numerical study of the Bayesian collapsed Tucker model using simulations and a real data example. Section 6 gives further discussion of results and implications.

\section{Notation and background}
\subsection{Log-linear models}
A standard approach to contingency table analysis parametrizes $\pi$ as a log-linear model satisfying certain constraints. 
For a subset of variables $E \subset V$, we adopt the notation of \cite{massam2009conjugate} to denote by $\bfi_E$ the cells in the marginal $E$-table, so that $\bfi_E \in \m I_E := \bigtimes_{j \in E} \m I_j$. Let $\teie$ denote the interaction among the variables in $E$ corresponding to the levels in $\bfi_E$. With these notations, a log-linear model assumes the form
\begin{align} \label{eq:loglineargen}
\log(\pi_{\bfi}) = \sum_{E \subset V} \teie. 
\end{align}
As a convention, $\theta_{\emptyset}$ corresponds to $E = \emptyset$. To identify the model we choose the corner parameterization \cite{agresti2002categorical,massam2009conjugate}, which sets $\teie = 0$ if there exists $j \in E$ such that $i_j = 1$. In the binary setting ($d_j = 2$ for all $j$) with corner parametrization, any $E$ for which $\teie \ne 0$ must have every element of $\bfi_E$ equal to 2. In this case we will represent $\teie$ as $\theta_E$ since there is no ambiguity. When $d>2$, the notation $\theta_E$ refers to the collection of parameters $\{ \teie : \bfi_E \in \m I_E \}$, and $\theta_E = 0$ indicates $\teie = 0 \text{ for all }  \bfi_E \in \m I_E$. 

Let $\bs{\theta} = \{\teie : 1 \notin \bfi_E\}$ denote the collection of the free model parameters and $S_{\theta}$ denote the collection of nonzero elements of $\bs{\theta}$. A saturated model includes all free model parameters, whence $|S_{\theta}| = \prod_j d_j$. Although any model that is not saturated is technically sparse, when we refer to sparse log-linear models we have in mind settings where $|S_{\theta}| \ll \prod_j d_j$. We will be primarily concerned with how the degree and structure of sparsity affects the nonnegative tensor rank of $\pi$.


An attractive feature of log-linear models is that the parameters are interpretable as defining conditional and marginal independence relationships between the $y_j$'s. Particularly simple interpretations arise from the class of hierarchical log-linear models. A log-linear model is hierarchical \cite{massam2009conjugate, dellaportas1999markov, darroch1980markov} if for every $E \subset V$ for which $\theta_E = 0$, we have $\theta_F = 0$ for all $F \supseteq E$. Here we work with a more general class of log-linear models that contains hierarchical models. We refer to this class as weakly hierarchical. 

\begin{definition} \label{def:weaklyhierarchical}
A log linear model is weakly hierarchical when the following condition is satisfied: if $\teie = 0$ for $E \subset V$ and $\bfi_E \in \m I_E$, then $\theta_F(\bfi'_F) = 0$ for every $F \supseteq E$ and $\bfi'_F \in \m I_F$ such that $i'_j = i_j$ for all $j \in E$. 
\end{definition}

When $d_j = 2$ for all $j$, weakly hierarchical models and hierarchical models define identical subsets of log-linear models, but if any $d_j > 2$, the collection of hierarchical models is a proper subset of the collection of weakly hierarchical models. To see this, suppose a model is weakly hierarchical. Assume $\theta_E = 0$. Then, $\theta_E(\bfi_E) = 0$ for all $\bfi_E \in \mc{I}_E$. Let $F \supseteq E$. For any $\bfi'_F \in \m I_F$, $\theta_F(\bfi'_F) = 0$ by weak hierarchicality, since $\theta_E(\bfi'_E) =0$. Since $\bfi'_F$ is arbitrary, we must have $\theta_F = 0$, proving hierarchicality. 

The essential difference between hierarchical and weakly hierarchical models is illustrated by the following example. Let $V = \{1,2,3\}$ and $d_1=d_2=d_3=4$.  Suppose 
$$
S_{\theta} = \left\{ \theta_{\{1\}}(2), \theta_{\{2\}}(2), \theta_{\{3\}}(2), \theta_{\{1,2\}}(2,2), \theta_{\{1,3\}}(2,2), \theta_{\{2,3\}}(2,2), \theta_{\{1,2,3\}}(2,2,2) \right\}.
$$
In other words, any interactions that correspond to all variables in $E$ taking level 2 are nonzero, and all others are zero. This model is weakly hierarchical but not hierarchical. For a model to be hierarchical, the collection of nonzero parameters must be uniquely specified by a generator -- a collection of subsets of $V$. For weakly hierarchical models, some interactions corresponding to a single subset $E$ may be zero and others nonzero, so long as definition \ref{def:weaklyhierarchical} is satisfied.

\subsection{Tensor Factorization Models}
An alternative to log-linear models is latent structure analysis (\cite{stouffer1950measurement, gibson1955extension, lazarsfeld1968latent, anderson1954estimation, madansky1960determinantal, haberman1974log, goodman1974exploratory}), which assumes the $y_1,\ldots,y_p$ are conditionally independent given one or more latent class variables. In marginalizing out the latent class variables, one obtains a tensor decomposition of $\pi$.  Latent structure models inducing PARAFAC and Tucker decompositions are briefly reviewed below. 

\subsubsection{PARAFAC models}
An $m$-component non-negative PARAFAC decomposition \cite{harshman1970foundations} of a probability tensor $\pi$ is given by
\begin{align}
 \pi = \sum_{h=1}^m \nu_h \lambda^{(1)}_h \otimes \ldots \otimes \lambda^{(p)}_h =  \sum_{h=1}^m \nu_h \bigotimes_{j=1}^p \lambda^{(j)}_h,  \label{eq:parafacmrg}
\end{align}
where $\otimes$ denotes an outer product\footnote{$\{\bigotimes_{j=1}^p \lambda^{(j)}_h\}_{i_1, \ldots i_p} = \prod_{j=1}^p \lambda_{h i_j}^{(j)}$}, each $\lambda^{(j)}_h \in \Delta^{(d_j -1)}$ is an element of the $(d_j-1)$ dimensional simplex\footnote{$\Delta^{(r-1)} = \{x \in \bb R^r: x_j \geq 0 \, \forall  \, j, \, \sum_{j=1}^r x_j = 1 \}$}, and $\nu \in \Delta^{(m-1)}$. By constraining $\nu$ and the $\lambda_h^{(j)}$s to be probability vectors, it is ensured that the entries of $\pi$ are non-negative and sum to one. The vectors $\lambda^{(j)}_h$ are referred to as the arms of the tensor. 

A probabilistic PARAFAC decomposition (\cite{dunson2009nonparametric}) of $\pi$ can be induced by a single index latent class model
\begin{align}
&y_{j} \mid z  \stackrel{\text{ind.}} \sim \text{Multi}(\{1,\ldots,d_j\},\lambda^{(j)}_{z 1},\ldots,\lambda^{(j)}_{z d_j}), \nonumber \\
&\text{Pr}(z=h) = \nu_h, h = 1, \ldots, m. \label{eq:parafacZ}
\end{align}
Marginalizing over the latent variable $z$, we obtain expression (\ref{eq:parafacmrg}). 

Unlike matrices, there is no unambiguous definition of the rank of a tensor. A notion of tensor rank is derived restricting attention to PARAFAC expansions. The nonnegative PARAFAC rank of a nonnegative tensor $M$ is the minimal value of $m$ for which there exist nonnegative vectors $\tlam_h^{(j)}$ such that
\begin{align}\label{eq:nonnegrank}
 M = \sum_{h=1}^m \bigotimes_{j=1}^p \tlam_h^{(j)}.
\end{align}
We will denote the nonnegative PARAFAC rank of a tensor $M$ as $\rnkp(M)$. Every nonnegative tensor $M$ with dimension $d^p$ can be represented exactly by a $m$-term nonnegative PARAFAC expansion with $m \le d^{p-1}$, ensuring that the factorization is flexible enough to characterize any $\pi$. Note that in the case of probability tensors, the definition in \eqref{eq:nonnegrank} is equivalent to the minimum $m$ such that (\ref{eq:parafacmrg}) holds, since the weights $\nu_h$ can be absorbed into the arms $\lambda^{(j)}_h$.

\subsubsection{Tucker models}
An $m$-component nonnegative Tucker decomposition \cite{tucker1966some, de2000multilinear} alternatively expresses the entries in $\pi$ as
\begin{align}
 \pi_{c_1\ldots c_p} = \sum_{h_1=1}^m \ldots \sum_{h_p=1}^m \phi_{h_1\ldots h_p} \prod_{j=1}^p \lambda_{h_j c_j}^{(j)},  \label{eq:tucker}
\end{align}
where $\phi$ is an $m^p$ {\em core} probability tensor and $\lambda^{(j)}_h \in \Delta^{d_j-1}$ for every $h$ and $j$. The Tucker decomposition can be thought of as a weighted sum of $m^p$ tensors each having PARAFAC rank one with weights given by the entries in $\phi$; conversely, the PARAFAC is a special case of the Tucker decomposition where the core is an $m \times 1$ probability vector. 

A probabilistic Tucker expansion of a probability tensor $\pi$ can be induced by a latent class model with a vector of latent class indicators $z = (z_{1}, \ldots, z_{p})$,
\begin{align}
&y_{j} \mid z  \stackrel{\text{ind.}} \sim \text{Multi}(\{1,\ldots,d_j\},\lambda^{(j)}_{z_j 1},\ldots,\lambda^{(j)}_{z_j d_j}), \nonumber \\
&\text{Pr}(z_1=h_1, \ldots, z_p = h_p) = \phi_{h_1 \ldots h_p}. \label{eq:TuckerZ}
\end{align}
See \cite{bhattacharya2012simplex} for a class of hierarchical models that induce a structured Tucker decomposition of a probability tensor. 

The Tucker decomposition gives rise to an alternative definition of the nonnegative tensor rank of a tensor $M$ as the minimal value of $m$ such that $M$ can be expressed exactly by an expansion of the form in (\ref{eq:tucker}). We will denote the nonnegative Tucker rank of a tensor $M$ as $\rnkt(M)$. In the case where $d_j = d$ for all $j$, $\rnkt(M) \le d$. The scale of Tucker ranks is quite different from that of PARAFAC ranks because the core itself has dimension $m^p$.

\section{Main results: Tensor rank of sparse log-linear models} \label{sec:generalpd}
\subsection{PARAFAC rank result for general $p$ and $d$} \label{subsec:rnkproofsgen}
We now provide bounds on the non-negative PARAFAC rank of joint probability tensors. There are few results on ranks of tensors beyond three dimensions and the techniques developed here are likely to be of independent interest. All proofs are deferred to the Appendix. In addition to the bounds developed in this section based on probabilistic arguments, we provide algebraic constructions in the two-dimensional case in a supplementary document. 

In the results that follow, we exploit the fact that a PARAFAC expansion of a probability tensor has a dual representation as a latent variable model \eqref{eq:parafacZ}, and the PARAFAC rank of a probability tensor can be defined in terms of the support of the corresponding latent class variable. Remark \ref{rem:limcomon} re-expresses an observation from \cite{lim2009nonnegative} that formalizes this relationship. 
\begin{remark} \label{rem:limcomon}
Suppose $\pi$ is a $\prod_{j=1}^p d_j$ probability tensor, and let $y_1,\ldots,y_p$ be categorical random variables with joint distribution defined by $\pi$. Then $\rnkp(\pi) = \bigwedge_{z \in \mc{Z}} |\spt(z)|$, where $\mc{Z}$ is the collection of all finitely-supported, discrete latent variables $z$ such that
\begin{align}
Pr(y_1=i_1,\ldots,y_p =i_p | z = h) = \prod_{j=1}^p Pr(y_j=i_j | z = h), \label{eq:lvind}
\end{align}
for all $h \in \spt(z)$ and $\bfi \in \m I_V$. 
\end{remark}
Therefore, if there exists a latent variable $z$ such that (\ref{eq:lvind}) is satisfied, then the rank of $\pi$ can be at most $|\spt(z)|$. Our recipe to create such discrete random variables $z$ is to partition the measure space $\mathcal{Y}$ on which $(y_1, \ldots, y_p)$ is defined, with a level of $z$ introduced for each set in the partition. As a convention to simplify notations, for subsets $B_j \subset \m I_j$, we shall identify the event $\{y_1 \in B_1, \ldots, y_p \in B_p\}$ with the event $\bigtimes_{j=1}^p B_j$ in the discrete $\sigma$-algebra generated by $\m I_V$. With this convention, $\m Y$ is identified with $\m I_V$, so that we can avoid making references to the abstract measure space $\m Y$ and instead partition $\m I_V$. For a partition $\m P$ of $\m I_V$, and $\{A_1, \ldots, A_{|\m P|}\}$ denoting an (arbitrary) enumeration of the sets in $\m P$, we define a discrete random variable $z = z_{\m P}$ corresponding to $\m P$ as 
\begin{align}\label{eq:lat}
z = h\ind_{A_h}, \, h = 1, \ldots, |\m P|.
\end{align}
Clearly, for any $z$ as in \eqref{eq:lat}, \eqref{eq:lvind} is equivalent to 
\begin{align}
Pr(y_1=i_1,\ldots,y_p =i_p | A_h) = \prod_{j=1}^p Pr(y_j=i_j | A_h), \label{eq:lvind1}
\end{align}
for all $h = 1, \ldots, |\m P|$ and $\bfi \in \m I_V$. 
In particular, for partitions $\m P_j$ of $\m I_j$, we can define the product partition $\m P$ as
\begin{align}\label{eq:part}
\m P = \bigtimes_{j=1}^p \m P_j := \left\{ \bigtimes_{j=1}^p B_j : B_j \in \m P_j \right\}.
\end{align}
It follows from properties of the cartesian product that $\m P$ indeed forms a partition of $\m I_V$ and $|\m P| = \prod_{j=1}^p |\m P_j|$. 

We now proceed to create random variables $z$ that satisfy \eqref{eq:lvind} (or equivalently partitions $\m P$ satisfying \eqref{eq:lvind1}). 
To begin with, 
observe that \\$Pr(y_1 = i_1, \ldots, y_p = i_p)$ 
\begin{align*}
& =  Pr(y_1 = i_1 \mid y_2 = i_2, \ldots, y_p = i_p) Pr(y_2 = i_2, \ldots, y_p = i_p) \\
& = \sum_{c_2 \in \m I_2} \ldots \sum_{c_p \in \m I_p} Pr(y_1 = i_1 \mid y_2 = c_2, \ldots, y_p = c_p) \ind_{(i_2 = c_2, \ldots, i_p = c_p)} Pr(y_2 = c_2, \ldots, y_p = c_p).
\end{align*}
Therefore, one such $z$ can be obtained by setting $\m P_1 = \m I_1$ and $\m P_j = \{ \{c_j\} : c_j \in \m I_j\}$ for $j \geq 2$, so that the event $\{z = h\}$ for each $h$ designates an event of the form $\m I_1 \times \{c_2\} \times \ldots \times \{c_p\}$. Clearly, this specification yields the trivial upper bound of $d^{p-1}$ to the PARAFAC rank when all $d_j = d$. Our main target is to show that much tighter bounds can be achieved under the assumption of weak hierarchicality. 

We introduce some additional notation here. For any $\bft$, we represent the nonzero two-way or higher interaction terms as $C_{\theta} = \{(E, \bfi_E) : |E| \geq 2, \teie \ne 0\}$ and denote the subclass of non-zero two-way interactions as $C_{\theta, 2} = \{(E, \bfi_E) : |E| = 2, \teie \ne 0\}$. For a variable $j \in V$, let $C_{\theta}^{(j)}$ denote the levels of variable $j$ that share a non-zero two-way or higher order interaction with at least one other variable. For weakly hierarchical models, it is sufficient to only search over the non-zero two-way interactions, so that $C_{\theta}^{(j)} = \{ c_j \in \m I_j : \text{ there exists } j' \neq j \text{ and } c_{j'} \in \m I_{j'} \text{ such that } \theta_{\{j,j'\}}(c_j,c_{j'}) \neq 0\}$.

If the model is weakly hierarchical, it follows from definition \ref{def:weaklyhierarchical} that for any subset $C'$ of $(\intlevj{j})^c$, $y_j \ind_{C'} \ci y_{[-j]}$, where $y_{[-j]} = (y_1,\ldots y_{j-1},y_{j+1},\ldots,y_p)$. Thus, instead of having to let the levels of $z$ vary over all events of the form $\big \{ \{c_2\} \cap \ldots \cap \{c_p\} \big \}$, one can coarsen the partition $\m P$ in \eqref{eq:part} by pooling together all the levels in $(\intlevj{j})^c$ to form a single element of $\m P_j$. Further improvement can be achieved by scanning through the variables in a particular order and only considering the subset of $\intlevj{j}$ that correspond to non-zero two-way interactions with variables that appear later in the ordering. We formalize the discussion in Theorem \ref{thm:rankgeneral} below. 

\begin{theorem} \label{thm:rankgeneral}
Suppose $\pi$ is a $d^p$ probability tensor generated by a weakly hierarchical log-linear model. Let $\sigma$ be a permutation on $V$.  For each $j = 1, \ldots, p-1$, denote $G_{\sigma}^{(j)} = \{\sigma(j+1), \ldots, \sigma(p)\}$ and define $B_{\sigma(j)}$ to be the following subset of $\intlevj{j}$:
\begin{align*}
B_{\sigma(j)} = \{ i_{\sigma(j)} \in \m I_{\sigma(j)} : \text{ there exists } f \in G_{\sigma}^{(j)} \text{ and } i_f \in \m I_f \text{ such that } \theta_{\{\sigma(j), f\}}(i_{\sigma(j)},i_f) \ne 0\}.
\end{align*}
Then, the rank of $\pi$ is at most
\begin{align*}
\bigwedge_{\sigma} \prod_{j=1}^{p-1} \left(\mid B_{\sigma(j)} \mid + 1 \right).
\end{align*}
\end{theorem}

The bound in Theorem \ref{thm:rankgeneral} gives the correct upper bound $d^{p-1}$ when the model is saturated, since then for any permutation $\sigma$ we have $|B_{\sigma(j)}| = (d-1)$ for $j = 1, \ldots, p-1$. More importantly, there is a massive reduction in the upper bound when the model is sparse; see Section \ref{subsec:rankcors} for multiple applications of Theorem \ref{thm:rankgeneral}. The bound is easy to compute and provides a useful estimate on the growth rate of the PARAFAC rank as $d$ and/or $p$ increases when the non-zero interactions are {\em uniformly} spread. However, if the interactions are highly structured, Theorem \ref{thm:rankgeneral} may yield the trivial upper bound irrespective of the true rank, as seen in Example \ref{ex:boundtrivial} below. 

Our next result provides sharper bounds on the PARAFAC rank. In the first part of Theorem \ref{thm:rankgeneraltight}, we provide a ``dimension-free'' upper bound that isn't affected by increasing $d$ as long as the true PARAFAC rank is constant. We then present a tight upper bound in the second part of Theorem \ref{thm:rankgeneraltight} which cannot be globally improved in the class of weakly hierarchical log-linear models. 


\begin{theorem} \label{thm:rankgeneraltight}
 Suppose $\pi$ is a probability tensor corresponding to a weakly hierarchical log-linear model. Let $\indset = \{\indsetj{1},\ldots,\indsetj{p} \}$ denote collections of sets of indices, where each $\indsetj{j} \subset \mc{I}_j$. Given $H$, define $T_{(C_{\theta},H)} = \{(E, \bfi_E) \in C_{\theta} : i_j \in H_j \text{ for some } j \in E \}$ and let
\begin{align}\label{eq:msH}
\ms{H} = \{H : T_{(C_{\theta},H)} = C_{\theta}\}.
\end{align}
Assume $ C_{\theta}^{(j)} \neq \emptyset$ for all $j$. Then,
\begin{align}\label{eq:tighter}
 \rnkp(\pi) \le \bigwedge_{H \in \ms{H}} \left( \prod_{j \in V} (\mid H_j \mid +1)\right).
\end{align}
For any $l \in V$, set $W_l = \{j \in V \setminus \{l\} : \mid H_j \mid = d-1\}$ and $\bar{W}_l = V \setminus W_l$. Then, a tight upper bound on $\rnkp(\pi)$ is 
\small{\begin{align}\label{eq:equality}
 \bigwedge_{H \in \ms{H}} \bigwedge_{l \in V} \left( \prod_{j \in V} (\mid H_j \mid +1) - \bigg[ \prod_{j \in W_l} (\mid H_j \mid +1) \bigg] \bigg[ \prod_{j \in \bar{W}_l} \mid H_j \mid \bigg] \right).
\end{align}}
\end{theorem}

\begin{remark}\label{rem:msH}
By definition, $T_{(C_{\theta},H)} \subset C_{\theta}$, hence the condition $T_{(C_{\theta},H)} = C_{\theta}$ in the definition of $\ms{H}$ in \eqref{eq:msH} equivalently requires that for every $(E, \bfi_E) \in C_{\theta}$, $i_j \in H_j$ for some $j \in E$. Moreover, for weakly hierarchical models, $T_{(C_{\theta},H)} = C_{\theta} \Leftrightarrow T_{(C_{\theta,2},H)} = C_{\theta,2}$.
\end{remark}

\begin{remark}\label{rem:VU}
Theorem \ref{thm:rankgeneraltight} assumes $ C_{\theta}^{(j)} \neq \emptyset$ for all $j$, i.e., every variable shares at least one second order interaction. Clearly, the set of variables which do not satisfy the condition are marginally independent of all other variables and do not contribute to the rank. Letting $U = \{j : C_{\theta}^{(j)} = \emptyset\}$, the statement of Theorem \ref{thm:rankgeneraltight} will continue to hold without this assumption as long as we replace all instances of $V$ by $V^* = V \setminus U$. 
\end{remark}

The main strategy of proving Theorem \ref{thm:rankgeneraltight} is once again to carefully construct a partition $\m P$ of $\m I_V$ and define $z$ as in \eqref{eq:lat}. For the first part of Theorem \ref{thm:rankgeneraltight}, we construct $\m P$ as in \eqref{eq:part}, with $\m P_j$ consisting of the singleton sets $\{c_j\}$ for $c_j \in H_j$ and the set $\bar{H}_j = \m I_j \backslash H_j$. It is then immediate that $|\m P_j| = |H_j| + 1$ and hence $|\m P| = \prod_{j=1}^p (|H_j| + 1)$. Denoting this partition generated by $H$ to be $\m P_H^0$ for future reference, the non-trivial part is to show that for any $H \in \ms{H}$, $y_1, \ldots, y_p$ are conditionally independent given any set $A$ in $\m P_H^0$.
The tight upper bound in the second part of Theorem \ref{thm:rankgeneraltight} exploits that certain sets in $\m P_H^0$ can be merged without sacrificing the conditional independence. Indeed, in all the examples where we could calculate the exact rank explicitly, the bound in \eqref{eq:equality} produces the exact rank. However, to show that \eqref{eq:equality} provides the exact rank, we need an additional condition; see Remark \ref{rem:eq} below.

\begin{remark}\label{rem:eq}
Let $\pi$ be a probability tensor corresponding to a weakly hierarchical log-linear model. Suppose for every $H \in \ms{H}$ for which there exists $H^* \in \ms{H}$ such that $H^*_j \subseteq H_j$ for every $j$, the smallest partition $\mc{A}_H^{\inf}$ satisfying (\ref{eq:lvind}) that can be formed from unions of the events in $\mc{A}_H^0$ satisfies $|\mc{A}_H^{\inf} | \ge |\mc{A}_{H^*}^{\inf}|$. Then (\ref{eq:equality}) gives the exact value of $\rnkp(\pi)$.
\end{remark}

The calculation of the expressions in \eqref{eq:tighter} and \eqref{eq:equality} can be more complicated than the bound in Theorem \ref{thm:rankgeneral}. Studying the computational complexity associated with calculating these expressions is left as a topic for further research. However, we can explicitly calculate these expressions in the setting of Example \ref{ex:boundtrivial} below and illustrate the improvement over Theorem \ref{thm:rankgeneral}. \footnote{Note that here and in several later examples, we assume that the main effects $\{\teie : |E| = 1\}$ are zero for convenience. While formally these models are not weakly hierarchical, the inclusion of nonzero main effects is irrelevant for the calculation of tensor ranks in these examples.}

\begin{example}\label{ex:boundtrivial}
Suppose $p = 2$ and $d_1 = d_2 = d$. Assume $\theta_{\{1,2\}}(2,c_2) \neq 0$ for all $c_2 \geq 2$, $\theta_{\{1,2\}}(c_1,2) \neq 0$ for all $c_1 \geq 2$ and $\theta_{\{1,2\}}(c_1,c_2) = 0$ otherwise. Thus, level $2$ of variable $1$ interacts with all levels except $1$ of variable $2$, and similarly, level $2$ of variable $2$ interacts with all levels except $1$ of variable $1$. In addition, for convenience of illustration, also assume that all main effects are zero, whence 
$$ \log \pi_{i_1 i_2} = \theta_0 + \theta_{\{1,2\}}(i_1,i_2) \ind_{(i_i \geq 2 \text{ or } i_2 \geq  2)}. $$
Letting $J_d$ denote the $d \times d$ matrix of all ones, we can write $\pi = e^{\theta_0} J_d + \tilde{\pi}$, where $\tilde{\pi}$ is a $d \times d$ non-negative matrix with all entries except for the second row or column equaling zero,
$$\tilde{\pi}_{i_1 i_2} = e^{\theta_{\{1,2\}}(i_1,i_2)} \ind_{(i_i \geq 2 \text{ or } i_2 \geq  2)}.$$
In case of nonnegative matrices, $\rnkp(A)$ equals the ordinary matrix rank $\mbox{rnk}(A)$ when $\mbox{rnk}(A) \leq 2$ (see \cite{gregory1983semiring}). It is easy to see that the ordinary matrix rank of $\tilde{\pi}$ is $2$, since there are at most two linearly independent columns. Hence, $\rnkp(\tilde{\pi}) = 2$ and 
applying Lemma \ref{lem:hadamard} in the Appendix, we conclude $\rnkp(\pi) \leq 1 + \rnkp(\tilde{\pi}) \leq 3$. Barring pathological cases, the ordinary rank $\mbox{rnk}(\pi)$ will always be $3$, and since $\rnkp(A) \geq \mbox{rnk}(A)$ for matrices (\cite{cohen1993nonnegative}), $\rnkp(\pi)$ will also be exactly $3$. 

In applying Theorem \ref{thm:rankgeneral}, we have $|B_1| = |B_2| = d-1$, so that we always get the trivial upper bound $d$ irrespective of the choice of $\sigma$. 

Next, apply Theorem \ref{thm:rankgeneraltight}. Observe that $H = \{\{2\}, \{2\}\} \in \ms{H}$, since all of the interaction terms have either $c_1 = 2$ or $c_2 = 2$, and hence the upper bound in \eqref{eq:tighter} is reduced to $4$ irrespective of the value of $d$. With this choice of $H$, the expression inside the minimum in \eqref{eq:equality} becomes $(|H_1|+1)(|H_2|+1) - |H_1||H_2| = 4-1 =3$, which returns the exact rank.  


\end{example}

Although a detailed proof of Theorem \ref{thm:rankgeneraltight} is provided in the appendix, we highlight the salient features of the proof through the following Example \ref{ex:illustrate}, which is an extension of Example \ref{ex:boundtrivial} to higher dimensions with a more complicated interaction structure. 
\begin{example}\label{ex:illustrate}
Suppose $p = 3$, $d\geq 2$ arbitrary, and the non-zero interactions are
\begin{align*}
&\theta_{\{1,2\}}(2, c_2) \ne 0 \text{ for all } c_2 \ge 2, \,\theta_{\{2,3\}}(2,c_3) \ne 0 \text{ for all } c_3 \ge 2, \,  \\
&\theta_{\{1,3\}}(c_1,3) \ne 0 \text{ for all } c_1 \ge 2.
\end{align*}
As in the previous example, all (except level $1$) levels of all variables interact in a structured way. Theorem \ref{thm:rankgeneral} will again produce the trivial bound $d^2$ for all $3! = 6$ permutations. However, letting $H = \{\{2\}, \{2\}, \{2\}\}$, clearly $H \in \ms{H}$ and hence the PARAFAC rank is no more than $2^3 = 8$ by \eqref{eq:tighter}.

With this choice of $H$, consider the event $A = \{2\} \times \bar{H}_2 \times \bar{H}_3$ in $\m P_H^0$, i.e., the event corresponding to $\{y_1 = 2, y_2 \neq 2, y_3 \neq 2\}$. Fix a cell $\bfi = (2, 4, 4)$. We show that \eqref{eq:lvind1} holds with $\bfi$ and $A$, i.e.,
\begin{align}\label{eq:ci_cond}
Pr(y_1 = 2, y_2 = 4, y_3 = 4 \mid A) = Pr(y_1 = 2 \mid A) Pr(y_2 = 4 \mid A) Pr(y_3 = 4 \mid A)
\end{align}
Following our convention of identifying sets in $\m Y$ and $\m I_V$, let $A^* = \{2\} \times \{4\} \times \{4\} \in \m I_V$ denote the set corresponding to $\{y_1 = 2, y_2 = 4, y_3 = 4\}$. Clearly, $A^* \subset A$, and hence $Pr(A^* \mid A) = Pr(A^*)/Pr(A)$ in the left hand side of \eqref{eq:ci_cond}. 

Observe that $Pr(y_1 = 2 \mid A) = 1$. Next,
\begin{align*}
Pr(y_2 = 4 \mid A) = \sum_{c_3 \neq 2} \frac{\pi_{2 4 c_3}}{Pr(A)} = Pr(A^* \mid A) \sum_{c_3 \neq 2} \frac{\pi_{2 4 c_3}}{\pi_{2 4 4}}.
\end{align*}
Similarly,
\begin{align*}
Pr(y_3 = 4 \mid A) = \sum_{c_2 \neq 2} \frac{\pi_{2 c_2 4}}{Pr(A)} = Pr(A^* \mid A) \sum_{c_2 \neq 2} \frac{\pi_{2 c_2 4}}{\pi_{2 4 4}}.
\end{align*}
Thus, \eqref{eq:ci_cond} is equivalent to showing
\begin{align}\label{eq:ci_cond_eq}
\frac{Pr(A)}{Pr(A^*)} = \sum_{c_2 \neq 2} \sum_{c_3 \neq 2} \frac{\pi_{2 4 c_3}}{\pi_{2 4 4}} \frac{\pi_{2 c_2 4}}{\pi_{2 4 4}}.
\end{align}
Next, by definition, 
\begin{align}\label{eq:ci_cond_eq1}
\frac{Pr(A)}{Pr(A^*)} = \sum_{c_2 \neq 2} \sum_{c_3 \neq 2} \frac{\pi_{2 c_2 c_3}}{\pi_{2 4 4}} = \sum_{c_2 \neq 2} \sum_{c_3 \neq 2} \frac{\pi_{2 c_2 c_3}}{\pi_{2 c_2 4}} \frac{\pi_{2 c_2 4}}{\pi_{2 4 4 }}.
\end{align}
To complete the argument, compare the right hand sides of \eqref{eq:ci_cond_eq} and \eqref{eq:ci_cond_eq1}, and note that for any $c_2 \neq 2$, 
\begin{align*}
\frac{\pi_{2 4 c_3}}{\pi_{2 4 4}} &= 
\frac{ \exp\{\theta_0 + \theta_{\{1,2\}}(2,4) + \theta_{\{1,3\}}(2,c_3) \} }{ \exp\{\theta_0 + \theta_{\{1,2\}}(2,4) + \theta_{\{1,3\}}(2,4) \}}  \\
&=\frac{ \exp\{\theta_0 + \theta_{\{1,2\}}(2,c_2) + \theta_{\{1,3\}}(2,c_3) \} }{ \exp\{\theta_0 + \theta_{\{1,2\}}(2,c_2) + \theta_{\{1,3\}}(2,4) \}}
= \frac{\pi_{2 c_2 c_3}}{\pi_{2 c_2 4}}
\end{align*}

\medskip
We now provide the main intuition behind the expression in \eqref{eq:equality}. Fix $l = 3$. Observe that the partition $\m P_H^0$ contains the sets $\{2\} \times \{2\} \times \{2\}$ and $\{2\} \times \{2\} \times \{1, 3, \ldots, d\}$. Merge these two sets to create a new partition which now has $7$ elements instead of the $8$ in $\m P_H^0$. Following the argument in the display after \eqref{eq:part}, we have conditional independence given $\{2\} \times \{2\} \times \m I_3$. Since this is the only set in the new partition that is not in $\m P_H^0$, the new partition satisfies \eqref{eq:lvind1}. Therefore, we see that it is possible to merge certain sets in $\m P_H^0$ to create a coarser partition that continues to satisfy \eqref{eq:lvind1}. More specifically, we merge those sets which have $(|V| - 1)$ coordinate projections that are singleton sets; see the Appendix for the proof in the general setting. 
\end{example}

Example \ref{ex:illustrate} is a simple yet non-trivial example of the general principles underlying Theorem \ref{thm:rankgeneraltight}. We provide another example with greater degree of complexity in the supplemental document. 

The expansions in Theorems \ref{thm:rankgeneral} and \ref{thm:rankgeneraltight} consist of nearly sparse tensors, a somewhat surprising result.  The latent variable $z$ in Theorem \ref{thm:rankgeneraltight} is defined by events of the form $\{i_j\}$ and $H_j^c$. Conditioning on events of the form $\{i_j\}$ results in the $j^{th}$ arm of the corresponding term of the PARAFAC expansion consisting of a single $1$ in the $i_j$ entry and zeros elsewhere. As a result, when the true tensor is sparse, the PARAFAC expansions implicit in Theorems \ref{thm:rankgeneral} and \ref{thm:rankgeneraltight} consist of many nearly sparse terms and a few nearly dense terms that arise from the events $H_j^c$. This means that sparsity in the log-linear parametrization implies both reduced PARAFAC rank and sparsity in the tensor arms, an observation that we develop further in section \ref{sec:applications}.

\section{Additional rank results and corollaries} 
We now provide a number of corollaries to Theorem \ref{thm:rankgeneral} that provide insight into cases where a relatively low PARAFAC rank can be expected. These results make some additional assumptions about the support of the log linear model. As a basis for comparison across the different cases, we will consider a sequence of ``true'' models given by their parameter vectors $\bs{\theta}_n$ and corresponding probability tensors $\pi_n$, where either $d_n \to \infty$, $p_n \to \infty$, or both, representing cases of increasing dimension of $\pi$, and determine what these settings imply about the growth of the PARAFAC rank when the number of nonzero parameters in the log-linear model grows at a logarithmic rate in $p$ or $d$. This is a common paradigm in high-dimensional asymptotics (see e.g. \cite{nardi2012log}) and provides a rough indication of the extent of dimension reduction that is achievable with PARAFAC decompositions in different cases. 

\subsection{Corollaries of PARAFAC rank results} \label{subsec:rankcors}

Corollary \ref{cor:fewlevels} shows that when the maximum number of interacting levels of all variables is small relative to $d$ the rank will be substantially reduced. 
\begin{corollary} \label{cor:fewlevels}
If $\mid \intlevj{j} \mid < m-1$ for all $j$, $\rnkp(\pi) < m^{p-1}$.
\end{corollary}
\begin{proof}
This follows immediately from Theorem \ref{thm:rankgeneral} by noting that the condition $\mid \intlevj{j} \mid < m-1$ implies that $\mid B_{\sigma(j)} \mid < m-1$ for every permutation $\sigma$ and every $j$.
\end{proof}
In the case where $m \ll d$, the condition in corollary \ref{cor:fewlevels} reduces the PARAFAC rank by a factor of $(d/m)^{p-1}$. In the setting where $d_n \to \infty$, with $m_n \asymp \log d_n$, we have $\rnkp(\pi_n) \asymp \log(d_n)^p$, and thus this condition would give asymptotically low-rank expansions when $d_n \to \infty$ with $p_n$ fixed. However, the PARAFAC rank still grows exponentially in $p_n$, so this assumption is unhelpful in controlling the growth of the PARAFAC rank when $p_n \to \infty$. By Theorem \ref{thm:rankgeneraltight}, the exact rank will also grow exponentially at the rate $p_n-1$, so in general the order of growth in $p_n$ of the exact PARAFAC rank is the same as that given by corollary \ref{cor:fewlevels}, which relies on Theorem \ref{thm:rankgeneral}.

If we also assume that certain types of conditional independence exist, useful bounds on the rate of growth in the PARAFAC rank in both $d$ and $p$ can be obtained. Corollary \ref{cor:condind} gives one such result.
\begin{corollary} \label{cor:condind}
Suppose that the conditions in corollary \ref{cor:fewlevels} hold and for $J \subset V$, set $y_{(J)} = \{y_j : j \in J\}$. Then if $y_{(J^c)}$ are independent given the variables $y_{(J)}$, $\rnkp(\pi) \le m^{|J|}$.
\end{corollary}
The simplest such case is represented graphically by example 1 in figure \ref{fig:graphs}: a single star-graph, where $y_7$ is the hub variable. More generally, the setting in corollary \ref{cor:condind} has a graphical representation where all edges involve at least one of the variables in $J$. The PARAFAC rank then grows exponentially in $|J|$, not $p$. If we take $d_n \to \infty$ and $p_n \to \infty$ with $m_n \asymp \log(d_n)$ and $|J_n| \asymp \log(p_n)$, we obtain $\rnkp(\pi_n) \asymp \log(d_n)^{\log p_n}$, so the growth in the rank slows to exponential in $\log p_n$.  

Similar bounds can be obtained when marginal independence exists, which is represented graphically in Example 2 in figure \ref{fig:graphs} and formalized in corollary \ref{cor:margind}. 
\begin{corollary} \label{cor:margind}
Suppose the conditions of corollary \ref{cor:fewlevels} hold, and suppose $J \subset V$ with $j \in J^c \Rightarrow y_j \ci y_{[-j]}$, and $|J|<p$. Then $\rnkp(\pi) \le m^{(|J|)}$.
\end{corollary}
Thus, under marginal independence, the PARAFAC rank will depend only on the number of variables that are not marginally independent; the same result that we obtained in corollary \ref{cor:condind} with conditional independence replaced by marginal independence. It follows we can also achieve the $\log(d_n)^{\log p_n}$ growth in the PARAFAC rank with the same assumptions on the growth of $m_n$ and $|J_n|$. We now show that a slight variation on the PARAFAC decomposition can eliminate the exponential factor of $\log(p_n)$ in cases where there are multiple marginally independent cliques with no separators. 

\subsection{Independent PARAFACs}
We first define a modified PARAFAC model. Divide $y_1,\ldots,y_p$ into $k$ groups, and let $s_j$ indicate the group membership of variable $j$. For each $s \in \{1,\ldots,k\}$ define a PARAFAC expansion for the marginal probability tensor corresponding to $\pi_s = Pr(\{ y_j  : s_j = s\})$, as
\begin{align*}
\pi^{(s)} = \sum_{h=1}^{m_s} \nu_{sh} \bigotimes_{j: s_j = s} \lambda^{(j)}_h.
\end{align*}
We define the joint distribution of $y_1,\ldots,y_p$ as 
\begin{align*}
\pi_{c_1,\ldots,c_p} = \prod_{s=1}^k \prod_{j : s_j = s} \pi^{(s)}_{c_j}.
\end{align*}
This model can be described succinctly as $k$ independent PARAFACs. This is a generalization of the sparse PARAFAC (sp-PARAFAC) model of \cite{zhou2013bayesian} to the case of more than two groups. This model is ideally suited to the case of a graphical structure with $k$ cliques and no separators, and gives much stronger control over parameter growth when the truth has such a structure, as shown in Theorem \ref{thm:indparafac}

\begin{theorem} \label{thm:indparafac}
Consider a sequence of hierarchical log-linear models for binary data defined by parameters $\bs{\theta}_n$, where $p_n \to \infty$. Let $\ms{F}_n$ be the collection of all cliques in the graphical representation of the model, and suppose $\mid \ms{F}_n \mid \asymp k_n$. Then if $\bigvee_{ F \in \ms{F}_n } \mid F \mid \asymp \log_2(p_n)$ and $F \cap F^* = \emptyset$ for all $F, F^* \in \ms{F}_n$ with $F \ne F^*$, the tensor $\pi_n$ can be expressed asymptotically by $k_n$ independent tensors $\pi_n^{(1)},\ldots,\pi_n^{(k_n)}$ with $\sum_{s=1}^{k_n} \rnkp(\pi_n^{(s)}) \asymp k_n p_n$.
\end{theorem}
An important conclusion is that in cases where the dependence has a particular structure, grouping variables and performing independent PARAFAC decompositions for each of the marginal probability tensors corresponding to the groups can reduce the effective number of parameters drastically and eliminate exponential scaling in $p_n$ altogether. The approach outlined above is limited to cases in which the graph has no separators -- i.e. the cliques are marginally independent. However, additional flexibility could be gained by introducing another set of parameters to control dependence between the groups, an approach that we ultimately propose through the collapsed Tucker model.

 \begin{figure}[ht]
\begin{tabular}{cccc}

\begin{tikzpicture}[scale=0.5,transform shape]
  \Vertex[x=0,y=0.5]{1}
  \Vertex[x=-1.5,y=1.5]{2}
  \Vertex[x=-1.5,y=3]{3}
  \Vertex[x=0,y=4]{4}
  \Vertex[x=1.5,y=3]{5}
  \Vertex[x=1.5,y=1.5]{6}
  \Vertex[x=0,y=2.25]{7}
  \tikzstyle{LabelStyle}=[fill=white,sloped]
  \Edge(1)(7)
  \Edge(2)(7)
  \Edge(3)(7)
  \Edge(4)(7)
  \Edge(5)(7)
  \Edge(6)(7)
  \draw (0,4.5) node[above]{$ \textsc{Ex. 1: Star graph}$};
\end{tikzpicture}

& \hspace{5mm} 
& \begin{tikzpicture}[scale=0.5,transform shape]
  \Vertex[x=0,y=0]{1}
  \Vertex[x=0,y=2]{2}
  \Vertex[x=2,y=2]{3}
  \Vertex[x=2,y=0]{4}
  \Vertex[x=4,y=1]{5}
  \Vertex[x=6,y=2]{6}
  \Vertex[x=6,y=0]{7}
  \tikzstyle{LabelStyle}=[fill=white,sloped]
  \Edge(1)(2)
  \Edge(2)(3)
  \Edge(3)(4)
  \Edge(4)(1)
  \Edge(1)(3)
  \Edge(2)(4)
  \draw (3,3) node[above]{$ \textsc{Ex. 2: Marginal independence}$};
\end{tikzpicture} 
& \\

\begin{tikzpicture}[scale=0.5,transform shape]
  \Vertex[x=0,y=0]{1}
  \Vertex[x=0,y=2]{2}
  \Vertex[x=2,y=2]{3}
  \Vertex[x=2,y=0]{4}
  \Vertex[x=4,y=1]{5}
  \Vertex[x=6,y=2]{6}
  \Vertex[x=6,y=0]{7}
  \tikzstyle{LabelStyle}=[fill=white,sloped]
  \Edge(1)(2)
  \Edge(2)(3)
  \Edge(3)(4)
  \Edge(4)(1)
  \Edge(1)(3)
  \Edge(2)(4)
  \Edge(5)(6)
  \Edge(6)(7)
  \Edge(7)(5)
  \draw (3,3) node[above]{$ \textsc{Ex. 3: Two separated cliques}$};
\end{tikzpicture} 
& \hspace{5mm} 
&\begin{tikzpicture}[scale=0.5,transform shape]
  \Vertex[x=0,y=0]{1}
  \Vertex[x=0,y=2]{2}
  \Vertex[x=2,y=2]{3}
  \Vertex[x=2,y=0]{4}
  \Vertex[x=4,y=1]{5}
  \Vertex[x=6,y=2]{6}
  \Vertex[x=6,y=0]{7}
  \tikzstyle{LabelStyle}=[fill=white,sloped]
  \Edge(1)(2)
  \Edge(2)(3)
  \Edge(3)(4)
  \Edge(4)(1)
  \Edge(1)(3)
  \Edge(2)(4)
  \Edge(4)(5)
  \Edge(5)(6)
  \Edge(6)(7)
  \Edge(7)(5)
  \draw (3,3) node[above]{$ \textsc{Ex. 4: Two cliques, one separator}$};
\end{tikzpicture} 
&\begin{tikzpicture}[scale=0.5,transform shape]
  \Vertex[x=0,y=0]{1}
  \Vertex[x=0,y=2]{2}
  \Vertex[x=2,y=2]{3}
  \Vertex[x=2,y=0]{4}
  \Vertex[x=4,y=2]{5}
  \Vertex[x=4,y=0]{6}
  \Vertex[x=6,y=2]{7}
  \Vertex[x=6,y=0]{8}
  \tikzstyle{LabelStyle}=[fill=white,sloped]
  \Edge(1)(2)
  \Edge(2)(3)
  \Edge(3)(4)
  \Edge(4)(1)
  \Edge(1)(3)
  \Edge(2)(4)
  \Edge(4)(5)
  \Edge(4)(6)
  \Edge(5)(6)
  \Edge(5)(7)
  \Edge(5)(8)
  \Edge(6)(7)
  \Edge(6)(8)
  \Edge(7)(8)
  \draw (3,3) node[above]{$ \textsc{Ex. 5: simulation example}$};
\end{tikzpicture}  \\

\end{tabular}
\caption{Example graphs} \label{fig:graphs}
\end{figure}
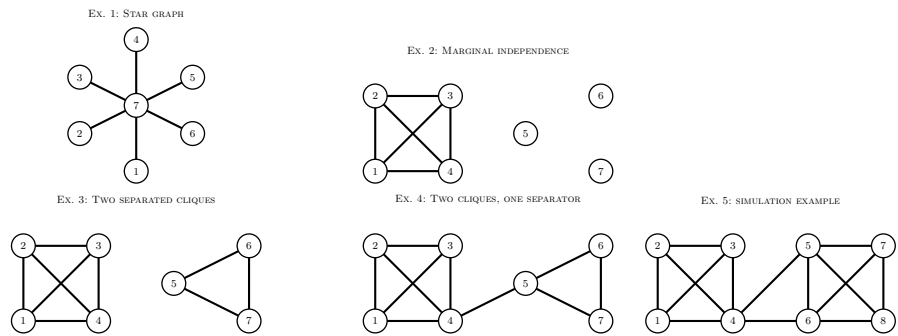

\subsection{Collapsed Tucker Models} \label{subsec:cTucker}
The results in section \ref{subsec:rnkproofsgen} and Theorem \ref{thm:indparafac} indicate that in many cases, grouping variables can substantially reduce the parameter complexity of tensor decompositions when the truth is a sparse log-linear model. In this section we propose a novel tensor factorization model that, like the independent PARAFACs discussed in section \ref{sec:generalpd}, groups variables, but is more flexible. A corollary to Theorem \ref{thm:rankgeneraltight} provides additional context and motivation for the proposed method.
\begin{corollary} \label{cor:tuckerrank}
If $\pi$ is a probability tensor corresponding to a sparse log-linear model then
\begin{align*}
\rnkt(\pi) \le \bigwedge_{H \in \ms{H}} \bigvee_{j \in V} (|H_j| + 1),
\end{align*} 
where $\ms{H}$ is the collection defined in the statement of Theorem \ref{thm:rankgeneraltight}.
\end{corollary}

The Tucker and PARAFAC decompositions represent two ends of the spectrum of tensor decomposition, with the PARAFAC having the simplest possible core but a rank that scales exponentially in $p$ (see Theorem \ref{thm:rankgeneraltight}), whereas the Tucker has a $m^p$ core but a rank that, by Corollary \ref{cor:tuckerrank}, does not depend on $p$ at all. One can conceptualize tensor decompositions that are intermediate in core complexity and rank scaling, which would correspond to more than one but fewer than $p$ latent variables. With this motivation, we propose a class of tensor factorizations that bridge the PARAFAC and Tucker approaches.  Specifically, we let
\begin{align}\label{eq:ctuck}
\pi_{c_1 \ldots c_p} = \sum_{h_1 = 1}^{m} \ldots \sum_{h_k = 1}^m \phi_{h_1 \ldots h_k} \prod_{j=1}^{p} \lambda_{h_j^* c_j}^{(j)},
\end{align}
where $h_j^* = h_{s_j}$ with $s_j \in \{1, \ldots,k\}$ for $j =1, \ldots, p$ and $k \ll p$ when $p$ is moderate to large.  The $s_j$'s are group indices for $\{y_j: j \in V\}$, with $s_j = \rho$ denoting that $y_j$ is allocated to group $\rho$. For a particular configuration of the $s_j$'s, $y_{[1:p]}$ are assigned to $k$ groups, and $s_j = s_{j'}$ indicates that $y_j$ and $y_{j'}$ belong to the same group. We refer to (\ref{eq:ctuck}) as a collapsed Tucker (c-Tucker) factorization.

Letting $z_i = (z_{i1}, \ldots , z_{ik})^{\T}$ denote a vector of group indices, the c-Tucker model in (\ref{eq:ctuck}) has a hierarchical representation where given $z_i$, $y_{[1:p]}$ are conditionally independent with $\mbox{pr}(y_{ij} = c_j \mid z_i, s_j) = \lambda_{z_{is_j} c_j}^{(j)}$.  If $k = 1$, we obtain the PARAFAC model (\ref{eq:parafacZ}) and for $k = p$ we have the Tucker factorization model in (\ref{eq:TuckerZ}).  When $1<k\ll p$, the core tensor $\phi$ in the c-Tucker has much smaller dimension relative to the Tucker core.

The c-Tucker model with $m=1$ represents the joint distribution of the variables as $k$ independent rank one PARAFAC expansions. For $m>1$, the model is a mixture of $m^k$ many $m$-term PARAFAC expansions. For any value of $k$, we define the c-Tucker rank of a tensor $\pi^{(0)}$ as the minimal value of $m$ such that there exists an exact $m$-term c-Tucker representation of $\pi^{(0)}$. We denote the nonnegative c-Tucker rank of a tensor $M$ where the number of variable classes is $k$ as $\rnkc(M,k)$. For any $1<k<p$, the nonnegative tensor ranks of a tensor $M$ obey the ordering $\rnkp(M) \ge \rnkc(M,k) \ge \rnkt(M)$.

Theorem \ref{thm:indparafac} shows that under certain circumstances, the c-Tucker rank can be dramatically less than the PARAFAC rank. When the data are binary and the graph consists of $k$ independent cliques, we can represent the joint distribution as $k$ independent PARAFAC expansions each with rank at most $2^{|\{h: s_j = h\}|-1}$. Since the c-Tucker model is a mixture of $m^k$ independent rank-$m$ PARAFAC expansions, the c-Tucker rank in this case is bounded above by $2^{|\{h: s_j = h\}| -1}$. 

The c-Tucker model is considerably more flexible than an independent PARAFAC model. In contrast to independent PARAFACs, interactions between variables in different groups exist in the c-Tucker model whenever $m > 1$. The PARAFAC expansions conditional on the groups parametrize within-group interactions, whereas the core parametrizes between-group interactions. Therefore the c-Tucker model can substitute a lower-rank core and larger groups for a higher-rank core and smaller groups. When the groups form cliques, the PARAFAC rank for each group will grow exponentially in group size. Thus in the case where there are few groups, the parameter complexity of a c-Tucker model will be dominated by the PARAFAC rank of the groups, whereas as the number of groups increases with constant $p$, the core dominates the parameter complexity. If variable groups are inferred, the tradeoff between a more Tucker-like and more PARAFAC-like model is automatic.

\section{Estimation and applications for c-Tucker models} \label{sec:applications}
We present an algorithm for inference and computation for c-Tucker models in the Bayesian paradigm, and provide guidance on prior choice. The model is illustrated in simulation studies and an application to the functional disability data from the national long term care survey (NLTCS). 

\subsection{Bayesian inference for c-Tucker models}
Bayesian inference for c-Tucker models requires that we specify priors on the parameters of the core, arms, and the group memberships of the variables. We choose conjugate Dirichlet priors on the arms $\lambda_{h_j^* c_j}^{(j)}$. The selection of prior hyperparameters on the arms is discussed in section \ref{subsec:armpriors}. To facilitate posterior computation, it is natural to model the probability tensor $\phi$ via a non-negative PARAFAC decomposition $\phi_{h_1 \ldots h_k} = \sum_{l = 1}^r  \xi_l \prod_{s=1}^k \psi_{l h_s}^{(s)}$, where $\xi = \{\xi_l\}$ is a vector of probabilities and $\psi_l^{(s)} = \{ \psi_{lh}^{(s)} \}$ are probability vectors of dimension $m$ for $s=\{1,\ldots,k\}$. We specify truncated Dirichlet process priors \cite{ishwaran2001gibbs} on the latent class probabilities $Pr(z_{is} = h)$ and fix the maximum number of latent classes such that one or more of the classes will be nearly unoccupied. A similar approach is used for the arms $\{\zeta_h^{(s)}\}$ in the PARAFAC expansion of the core. We choose a $\text{Dirichlet}(1/k,\ldots,1/k)$ prior on variable group probabilities. 

Our Bayes c-Tucker model can be expressed in hierarchical form as
\begin{align}\label{eq:hier_mod1}
& y_{ij} \mid z_{i1},\ldots,z_{ik}, \bfl^{(j)} \sim \mbox{Multinomial}(\{1,\ldots,d_j\}, \lambda_{z_{i h_{s_j}}1}^{(j)}, \ldots, \lambda_{z_{i h_{s_j}}d_j}^{(j)}), \notag \\
& \bfl_h^{(j)} \sim \mbox{Diri}(a_{h 1}, \ldots, a_{h d_j}), \notag \\
& z_{is} \mid w_i, \bfps^{(s)} \sim \mbox{Multinomial}(\{1,\ldots,m\}, \psi_{w_i1}^{(s)}, \ldots, \psi_{w_im}^{(s)}) \notag \\
& \mbox{pr}(w_i = l) = \nu_l^* \prod_{t < l} (1 - \nu_t^*), \, \nu_l^* \sim \mbox{beta}(1,\beta) \notag \\
& \psi_{lh}^{(s)} = \zeta_{lh}^{(s)} \prod_{h' < h} (1 - \zeta_{lh'}^{(s)}), \, \zeta_{lh}^{(s)} \sim \mbox{beta}(1,\delta_s) \notag \\
& s_1,\ldots,s_p \sim \mbox{Multinomial}(\{1,\ldots,k\}, \xi_1,\ldots,\xi_k) \notag \\
& \xi \sim \mbox{Dirichlet}(1/k,\ldots,1/k).
\end{align}
The joint likelihood of $(y_i, z_{i1},\ldots,z_{ik}, w_i)$ for $i = 1, \ldots, n$ given the model parameters $(\bfl, \bfps^{(1)}, \ldots,\bfps^{(k)}, \beta, \delta_1,\ldots, \delta_k)$ is given by
\begin{align}\label{eq:jnt_lik1}
& \left[\prod_{s=1}^ k \prod_{i=1}^n \prod_{j : s_j = s} \prod_{c_j=1}^{d_j} \left\{\lambda_{z_{is} c_j}^{(j)}\right\}^{\ind(y_{ij} = c_j)} \right] \times \notag \\
& \left[ \prod_{s=1}^k \prod_{i=1}^n  \prod_{h=1}^{m} \left\{\psi_{w_ih}^{(1)}\right\}^{\ind(z_{is} = h)} \right] \times \left[\prod_{i = 1}^n \prod_{l = 1}^k \nu_l^{1(w_i = l)} \right].
\end{align}
Bayesian computation for this model can be performed using a straightforward Gibbs sampler. The full conditionals and details of the computation are given in appendix \ref{app:computation}.

\subsection{Choices of priors on arms} \label{subsec:armpriors}
Theorems \ref{thm:rankgeneral} and \ref{thm:rankgeneraltight} establish that PARAFAC expansions of sparse log-linear models themselves consist of sparse terms. This suggests that the choice of $\text{Dirichlet}(1,\ldots,1)$ priors on the arms in a Bayes PARAFAC model, as in \cite{dunson2009nonparametric}, may not concentrate efficiently around probability tensors corresponding to sparse log-linear models. Because the c-Tucker is a mixture of PARAFACs, it is likely that the same principle operates in the c-Tucker expansion. Therefore we consider alternative priors on the tensor arms. Deriving explicit distributions transforming between the tensor and log-linear parameterizations is analytically intractable, so we conduct simulations to study the induced prior on the log-linear parameters for different choices of priors on tensor arms. 

The top panel of Figure \ref{fig:inducedprior} shows the prior induced on main effects, 2-way interactions, and 3-way interactions with a $\text{Dirichlet}(1,\ldots,1)$ prior on the arms, with $p=3$, $d=20$, and where the number of components is five ($m=5$). These histograms show the value of a single main effect, two-way interaction, and three-way interaction from 10,000 Monte Carlo samples. To place substantial prior probability on (approximately) sparse log-linear models, the induced prior on the interaction terms should be sharply peaked at zero with heavy tails. To empirically illustrate the benefit of nearly sparse arms in inducing such priors on log-linear model parameters, we sampled from a PARAFAC model in which the value of $a_h$ in the $\text{Dirichlet}(a_h,\ldots,a_h)$ prior on the tensor arms is decreasing in the component index $h$. Choosing $a_h = 1$ for $h=1$, $a_h = 1/d$ for $h = 2,3$, and $a_h = 1/d^2$ for $h=4,5$, we obtain the results in the second row of Figure \ref{fig:inducedprior}.

\begin{figure}[h]
\includegraphics[width=3in]{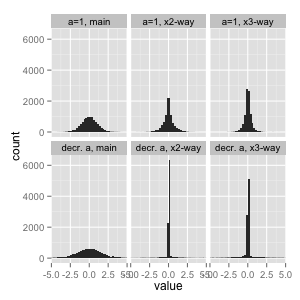} \\
\caption{Histograms of one main effect, one two-way interaction, and one three-way interaction for two priors on tensor arms in a Bayes PARAFAC model. In each row, the left panel shows the main effect, the center panel shows the two-way interaction, and the right panel shows the three-way interaction. The top row shows histograms for $a_h=1$ for all $h$, the second row shows histograms for $a_h$ decreasing with $h$. All histograms are based on 10,000 Monte Carlo draws. } \label{fig:inducedprior}
\end{figure}

Figure \ref{fig:inducedpriorL1} shows histograms of the $L_1$ norm of all of the main effects and interaction terms for the same choices of priors on the tensor arms. While the prior that sets $a_h=1$ for all $h$ induces a prior on the $L_1$ norm for interaction terms of both orders that has all of its mass bounded away from zero, the prior with $a_h$ decreasing in $h$ has a mode near zero in the $L_1$ norm of the two-way interactions. For three-way interactions, the mass is shifted toward zero for the prior with decreasing $a_h$. This confirms that the effect of concentrating the prior around nearly sparse terms in the PARAFAC is empirically significant in inducing shrinkage toward sparse log-linear models. We adopt similar Dirichlet priors with decreasing concentration parameters for the arms in the c-Tucker model.

\begin{figure}[h]
\includegraphics[width=3in]{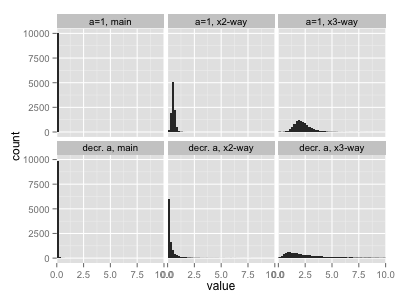} \\
\caption{Histograms of $L_1$ norm of main effects, two-way interactions, and three-way interactions for two priors on tensor arms in a Bayes PARAFAC model. In each row, the left panel shows the main effects, the center panel shows the two-way interactions, and the right panel shows the three-way interactions. The top row shows histograms for $a_h=1$ for all $h$, the second row shows histograms for $a_h$ decreasing with $h$. All histograms are based on 10,000 Monte Carlo draws.} \label{fig:inducedpriorL1}
\end{figure}

\subsection{Simulation studies and application}
Two simulations were conducted to illustrate the performance of the c-Tucker model. In both cases, all $y_j$ are binary. Posterior computation was performed using the MCMC algorithm described in appendix \ref{app:computation}. A burn-in of 10,000 iterations was performed, after which the MCMC was run for an additional 15,000 iterations, with samples gathered every tenth iteration. 

In the first case, $n=1000$ observations were simulated from the graphical model in example 3 of Figure \ref{fig:graphs}. The graph determines $S_{\theta}$, the elements of which are sampled \emph{iid} $N(0,9)$, and $\theta_{\emptyset}$ is calculated so that $\pi^{(0)}$ is a probability tensor. In this example $p=7$, $d=2$ and the contingency table has $128$ cells. The groups are specified \emph{a priori} as $s_j = 1$ for $j \le 4$ and $s_j = 2$ for $j > 5$. The posterior probability $Pr(\rnkp(\phi) > 1) < 0.001$, and thus the posterior correctly recovers the independence of the two cliques. Figure \ref{fig:fixedgroupsbox} shows a boxplot of posterior samples of $\theta_{E'}$ for all $E'$ with $|E'| = 2$.  These posterior samples are obtained by solving for the log-linear model parameters given posterior samples of $\pi$. The posterior concentrates around the true parameter values, but shrinks them somewhat toward zero, as expected. The two-way 
interactions that are truly zero have very narrow credible intervals centered at zero. 

\begin{figure}[h]
\includegraphics[width=4in]{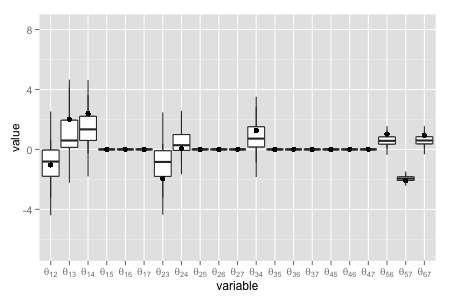} \\
\caption{Posterior samples and true parameter values for simulation of data with dependence structure given in the graph in example 3 of Figure \ref{fig:graphs}. The overlaid circles show the true parameter values, the dark horizontal lines the median, and the length of the whiskers shows a 95 percent posterior credible interval for each parameter.} \label{fig:fixedgroupsbox}
\end{figure}

In the second simulation example, we simulate $n=2000$ observations from the graphical model in example 5 of Figure \ref{fig:graphs}, where again the graph determines the support and nonzero elements of $\bs{\theta}$ are sampled $iid$ $N(0,3)$. In this example $p=8$ and $\pi^{(0)}$ has 256 cells, so the sample size relative to the parameter complexity is the same as in the first simulation. We learn the group identities of the variables with $k=3$. Table \ref{tab:topgroupssim} shows ten groupings with the highest posterior probabilities and their corresponding posterior weights. The posterior probability that exactly two of the groups are occupied is $0.89$. Notably, the groups do not in general correspond to cliques. Since the c-Tucker model parametrizes interactions between groups via the core and interactions within groups with PARAFAC expansions, one would not necessarily expect the variable groups to corresponding to cliques except when the cliques are marginally independent. 

Shown in table \ref{tab:pwcv_ex5} are the posterior means of pairwise Cram\'{e}r's V for $y_1,\ldots,y_8$, with the true values displayed for comparison. The posterior means are very similar to the true values. Figure \ref{fig:learngroupsbox} shows boxplots of posterior samples for the interaction terms in this simulation study. Also shown are the true parameter values (black circles) and parameter estimates obtained using lasso (blue circles). The lasso penalty was selected using 10-fold cross-validation and estimated using the \texttt{glmnet} package for \texttt{R} \cite{Friedman:Hastie:Tibshirani:2009:JSSOBK:v33i01}, with the main effects set to have no penalty. The 95 percent posterior credible intervals for the c-Tucker model have 96.5 percent coverage, and the c-Tucker model is at least as successful as lasso at recovering the true parameter values. The intervals are in some cases quite broad. We attribute this largely to the fact that the c-Tucker model, unlike commonly employed priors on log-linear 
models, does not restrict the induced log-linear model to be graphical, hierarchical, or even weakly hierarchical. It is undoubtedly the case that in many instances there exist values of $\bs{\theta}$ that do not satisfy weak hierarchicality, yet for which $\pi(\bs{\theta}) \approx \pi^{(0)}$, leading to very similar likelihood values. Therefore, the induced posterior for $\bs{\theta}$ in the c-Tucker model is likely a mixture of numerous weakly hierarchical and non-weakly hierarchical log-linear models, leading to the relatively broad credible intervals that we observe. 

\begin{table}[h]
\begin{tabular}{cccc}
\hline
\bf{group 1} & \bf{group 2} & \bf{group 3} & \bf{post. prob.} \\
\hline
  $\emptyset$ & $\{2 ,4,6,7,8\}$ & $\{1,3,5\}$ & 0.200 \\
  $\emptyset$ & $\{2,4,6,7\}$ & $\{1,3,5,8\}$ & 0.140 \\
  $\emptyset$ & $\{2,4,5,6,7,8\}$ & $\{1,3\}$ & 0.112 \\
  $\emptyset$ &  $\{2,4,5,6,7 \}$ & $\{1,3,8\}$ & 0.085 \\
  $\emptyset$ &  $\{2,4,7,8 \}$ & $\{1,3,5,6\}$ & 0.067 \\
  $\emptyset$ &  $\{2,4,7 \}$ & $\{1,3,5,6,8\}$ & 0.063 \\
  $\emptyset$ &  $\{2,4,5,7,8\}$ & $\{1,3,6\}$ & 0.038 \\
  $\emptyset$ &  $\{2,4,5,7\}$ & $\{1,3,6,8\}$ & 0.038 \\
  $\emptyset$ &  $\{2,4,6\}$ & $\{1,3,5,7,8\}$ & 0.037 \\
  $\emptyset$ &  $\{2,4,6,8\}$ & $\{1,3,5,7\}$ & 0.029 \\
  \hline
\end{tabular}
\caption{Table of variable groupings with ten highest posterior probabilities for simulation example with dependence structure given in example 5 of Figure \ref{fig:graphs}.} \label{tab:topgroupssim}
\end{table}

\begin{table}[ht]
\centering
\begin{tabular}{rrrrrrrrr}
  \hline
 & $\bs{y}_1$ & $\bs{y}_2$ & $\bs{y}_3$ & $\bs{y}_4$ & $\bs{y}_5$ & $\bs{y}_6$ & $\bs{y}_7$ & $\bs{y}_8$ \\ 
  \hline
$\bs{y}_1$ & 0.00 & 0.44 & 0.36 & 0.30 & 0.00 & 0.01 & 0.00 & 0.00 \\ 
$\bs{y}_2$ & 0.41 & 0.00 & 0.78 & 0.53 & 0.00 & 0.01 & 0.00 & 0.00 \\ 
$\bs{y}_3$ & 0.40 & 0.80 & 0.00 & 0.67 & 0.00 & 0.02 & 0.00 & 0.00 \\ 
$\bs{y}_4$ & 0.28 & 0.53 & 0.64 & 0.00 & 0.00 & 0.02 & 0.00 & 0.00 \\ 
$\bs{y}_5$ & 0.00 & 0.00 & 0.00 & 0.00 & 0.00 & 0.00 & 0.02 & 0.03 \\ 
$\bs{y}_6$ & 0.01 & 0.02 & 0.02 & 0.02 & 0.00 & 0.00 & 0.09 & 0.02 \\ 
$\bs{y}_7$ & 0.00 & 0.00 & 0.00 & 0.00 & 0.00 & 0.08 & 0.00 & 0.17 \\ 
$\bs{y}_8$ & 0.00 & 0.00 & 0.00 & 0.00 & 0.00 & 0.02 & 0.03 & 0.00 \\ 
   \hline
\end{tabular}
\caption{True pairwise Cram\'{e}r's V (above the main diagonal) and posterior mean pairwise Cram\'{e}r's V (below the main diagonal) in simulation example for graphical model in example 5 of figure \ref{fig:graphs}.} \label{tab:pwcv_ex5}
\end{table}

We apply the model to analysis of functional disability data from the national long term care survey (NLTCS). The data take the form of a $2^{16}$ contingency table, and are extensively described in \cite{dobra2011copula}. We performed posterior computation using the MCMC algorithm described in appendix \ref{app:computation}. After a burn-in period of 10,000 iterations, we collected samples every tenth iteration for 15,000 additional iterations. Table \ref{tab:pwcv_nltcs} shows the posterior means of pairwise Cram\'{e}r's V and $Pr(H_{1,\rho}|\bs{y})$, where $H_{1,\rho} = \ind(\rho>0.1)$ and $\rho$ is the pairwise Cram\'{e}r's V. For comparison, we reproduce the same results based on posterior samples for the copula Gaussian graphical model from \cite{dobra2011copula} in Table \ref{tab:pwcv_nltcs_copula}. Our results demonstrate close agreement with \cite{dobra2011copula}.

\section{Conclusion}
The relationship between the sparsity of a log-linear model and the rank of the associated probability tensor derived here is broadly applicable beyond the scope of this paper. There are few results on the rank of tensors of general dimensions and thus the proof techniques employed here may be of independent interest. In addition, there is clear need for additional work on prior choice in Bayesian latent structure models. The theoretical results obtained here provide a basis for high-dimensional asymptotic studies of latent structure models with sparse log-linear models as a truth class, and are also relevant for the development of optimization-based penalized likelihood approaches for inference in latent structure models. Moreover, there is a substantial computational burden in transforming between the two parametrizations unless the sparsity pattern in the log-linear model parameters can be determined directly from the tensor expansion. Further development of the relationship between sparsity and tensor rank could alleviate this important computational hurdle.

\appendix

\section{Proofs and auxiliary results} \label{app:proofs}

\subsection{Auxiliary results}
We state and prove lemma \ref{lem:hadamard} which is used to prove Theorem \ref{thm:rankgeneral}. 
\begin{lemma} \label{lem:hadamard}
 Let $\pi$ and $\psi$ be two non-negative $d^p$ tensors. Then, $\rnkp(\pi \circ \psi) \leq \rnkp(\pi) \rnkp(\psi)$, where $\circ$ denotes a Hadamard product, and $\rnkp(\pi + \psi) \leq \rnkp(\pi) + \rnkp(\psi)$. 
\end{lemma}
\begin{proof}
Let $\rnkp(\pi) = m, \rnkp(\psi) = k$ and $\phi = \pi \circ \psi$. For $1 \leq j \leq p$, there exist non-negative vectors $\lambda_h^{(j)} \in \bb{R}_{+}^d, h = 1, \ldots, m$ and $\zeta_l^{(j)} \in \bb{R}_+^d, l =1, \ldots, k$, such that $\pi = \sum_{h=1}^m \lambda_h^{(1)} \otimes \ldots \otimes \lambda_h^{(p)}$ and $\psi = \sum_{l=1}^k \zeta_h^{(1)} \otimes \ldots \otimes \zeta_h^{(p)}$.  Then, it is easy to see that 
$$
\phi = \sum_{h=1}^m \sum_{l=1}^k \gamma_{hl}^{(1)} \otimes \ldots \otimes \gamma_{hl}^{(p)},
$$
where $\gamma_{hl}^{(j)} = \lambda_h^{(j)} \circ \zeta_l^{(j)}$ for $1 \leq j \leq p$. Clearly, for any $j$, $\gamma_{hl}^{(j)} \in \bb{R}_{+}^d$ for $h = 1, \ldots, m; l = 1, \ldots, k$. Thus, $\rnkp(\phi) \leq m k$. 

In particular, if $\rnkp(\psi) = 1$, we have $\rnkp(\phi) \leq m$. This bound cannot be globally improved, or in other words, the upper bound can be achieved.  Take for example, $\psi = \zeta^{(1)} \otimes \ldots \otimes \zeta^{(p)}$, with $\zeta^{(j)} = (1, \ldots, 1)^{\T}$ for all $j$. 

Finally, we note that if
\begin{align*}
\pi = \sum_{h=1}^{m_1} \bigotimes_{j=1}^p \tlam^{(j)}_h \text{ and } \psi = \sum_{h=1}^{m_2} \bigotimes_{j=1}^p \tilde{\zeta}^{(j)}_h
\end{align*}
then 
\begin{align*}
\pi + \psi= \sum_{h=1}^{m_1} \bigotimes_{j=1}^p \tlam^{(j)}_h   + \sum_{h=1}^{m_2} \bigotimes_{j=1}^p \tilde{\zeta}^{(j)}_h
\end{align*}
so $\rnkp(\pi+\psi) = m_1 + m_2 = \rnkp(\pi) + \rnkp(\psi)$. 
\end{proof}

\subsection*{Proof of Theorem \ref{thm:rankgeneral}}
Without loss of generality, we assume $\sigma$ is the identity permutation and drop the corresponding subscripts. Let $\m P^{(1)}$ be the partition of $\m I_1$ consisting of the singleton sets $\{c\}$ for $c \in D^{(1)}$ and the set $(D^{(1)})^{c}$. Weak hierachicality ensures that $y_1 \ind_{(y_1 \in A)} \ci y_{2:p}$ for any $A \in \m P^{(1)}$, where $y_{2:p} = y_{2},\ldots,y_{p}$. Using the fact that for any two random variables $Z_1, Z_2$ and any measurable set $A$, $Z_1 \ind_{(Z_1 \in A)} \ci Z_2 \Leftrightarrow Z_1 \ci Z_2 \mid A$, we have $y_1 \ci y_{2:p} \mid A$ for any $A \in \m P^{(1)}$. Enumerating the sets in $\m P^{(1)}$ as $A_1, \ldots, A_{m_1}$, with $m_1 = |\m P^{(1)} | = |D^{(1)}| + 1$, we can write $\pi$ as 
\begin{align}\label{eq:cond_on_1}
\pi_{c_{1} \ldots c_{p}} = \sum_{h=1}^{m_{1}} \nu_h \lambda_{h c_{1}} \psi_{h c_{2} \ldots c_{p}},
\end{align}
where for each $1 \leq h \leq m_1$, $\nu_h = Pr(A_h)$, $\lambda_h \in \Delta^{(d-1)}$ with $\lambda_{hc} = Pr(y_1 = c \mid A_h)$ and $\psi_h$ is a $d^{p-1}$ non-negative tensor representing the joint probability of $y_{2:p} \mid A_h$, i.e.,
\begin{align*}
\psi_{hc_2 \ldots c_p} = Pr(y_2 = c_2, \ldots, y_p = c_p \mid A_h).
\end{align*}
Define $d^p$ tensors $\{\pi_h^{(1)}\}$ and $\{\pi_h^{(2)}\}$ by 
\begin{align*}
\pi^{(1)}_h &= \lambda_h \otimes \bs{1} \ldots \otimes \bs{1}  \\
(\pi^{(2)}_h)_{c_{1}\ldots c_{p}} &= \nu_h \psi_{h c_{2} \ldots c_{p}}. 
\end{align*}
The expansion of $\pi$ in \eqref{eq:cond_on_1} can now be written in tensor notation as $\pi = \sum_{h=1}^{m_1} \pi^{(1)}_h \circ \pi^{(2)}_h$. 
Clearly $\rnkp(\pi^{(1)}_h) = 1$ and it is easily verified that $\rnkp(\pi^{(2)}_h) \le \rnkp(\psi_h)$ for all $h$. Therefore, using Lemma \ref{lem:hadamard} we have that $\rnkp(\pi) \le m_{1} r$, where $r = \rnkp(\psi_h)$. 

Recursively applying this process for the variables $y_{2},\ldots,y_{p}$, we can show that $r \leq \prod_{j=2}^p m_{j} = \prod_{j=2}^p (|D^{(j)}| +1)$, so that 
$$\rnkp(\pi) \leq \prod_{j=1}^p \{ |D^{(j)}| + 1\}.$$
For any permutation $\sigma$, we can obtain a result as in the above display by scanning through the variables in the sequence $\sigma(1), \ldots, \sigma(p)$. Taking the minimum over all permutations $\sigma$, we obtain the desired result. 

\subsection*{Proof of (12) in Theorem \ref{thm:rankgeneraltight}}
Fix $H \in \ms{H}$. Let $\bar{H}_j = \m I_j \backslash H_j$ and let $\m P_{H, j}$ denote the partition of $\m I_j$ consisting of the singleton sets $\{i_j\}$ for $i_j \in H_j$ and the set $\bar{H}_j$. Define a partition $\m P_H^0$ of $\m I_V$ as the cartesian product of the partitions $\m P_{H, j}$ as in \eqref{eq:part}. We show that for any set $A \in \m P_H^0$, \eqref{eq:lvind1} is satisfied, i.e., 
\begin{align}\label{eq:lvind2}
Pr(y_1 = i_1, \ldots, y_p = i_p \mid A) = \prod_{j=1}^p Pr(y_j = i_j \mid A),
\end{align}
for any $\bfi \in \m I_V$. 
Based on the discussion in Section 3.1, the random variable $z = z_H^0$ corresponding to the partition $\m P_H^0$ defined via \eqref{eq:lat} will then satisfy \eqref{eq:lvind}, implying
$$\rnkp(\pi) \leq |\m P_H^0| = \prod_{j=1}^p | \m P_{H, j} | = \prod_{j=1}^p (|H_j| + 1). $$

We now proceed to establish \eqref{eq:lvind2}. Fix $A \in \m P_H^0$. By construction, 
\begin{align}\label{eq:A_def}
A = \bigtimes_{k \in \bar{J}} \{c_k\} \times \bigtimes_{j \in J} \bar{H}_j
\end{align}
for some $J \subset V$, $\bar{J} = V \backslash J$ and $c_k \in H_k$ for all $k \in \bar{J}$. Without loss of generality, we assume $J = \{q, \ldots, p\}$ for some integer $q \geq 1$. 

Let $\tilde{ \m I}_V$ denote the subset of $\m I_V$ consisting of cells $\bfi$ such that 
$i_k = c_k$ for all $k \in \bar{J}$ and $i_j \in \bar{H}_j$ for all $j \in J$. It is easy to see that for any $\bfi \notin \tilde{\m I}_V$, \eqref{eq:lvind2} is satisfied trivially since both sides are reduced to zero or one simultaneously. Hence, it suffices to show that \eqref{eq:lvind2} holds for any $\bfi \in \tilde{ \m I}_V$. 

Fix $\bfi \in \tilde{ \m I}_V$. Let $A_{\bfi}$ denote the subset of $\m I_V$ corresponding to the event $\{y_j = i_j, j \in V\}$ in $\m Y$, so that
$$A_{\bfi} = \bigtimes_{j \in V} \{i_j\}, \quad Pr(A_{\bfi}) = \pi_{\bfi}.$$
Clearly, $A_{\bfi} \subset A$, which implies $Pr(A_{\bfi} \mid A) = \pi_{\bfi}/Pr(A)$. Further, $Pr(y_k = i_k \mid A) = 1$ for any $k \in \bar{J}$, since $i_k = c_k$ for $k \in \bar{J}$. Therefore, \eqref{eq:lvind2} reduces to showing
\begin{align}\label{eq:lvind3}
\frac{\pi_{\bfi}}{Pr(A)} = \prod_{l \in J} Pr(y_l = i_l \mid A).
\end{align}
For $E \subset V$, we introduce the notation
$$\barh_E = \prod_{j \in E} \bar{H}_j. $$
We shall use $\bfa$ to generically denote an element of $\barh_J$, i.e., $\bfa$ is a $|J|$-vector of indices with $\alpha_j$ the entry in $\bfa$ corresponding to variable $j \in J$. For $l \in J$, $J^{(-l)}$ shall denote the set $J \backslash \{l\}$. We use $\bfa^{(l)}$ to generically denote an element of $\barh_{J^{(-l)}}$, with $\alpha_j^{(l)}$ the entry in $\bfa^{(l)}$ corresponding to variable $j \in J^{(-l)}$. 

Finally, for a partition of $V$ into $J_1, J_2, J_3$, denote
\begin{align}\label{eq:pi_altnot}
\pi^{(J_1, J_2, J_3)}_{f_j g_k h_l} := Pr \bigg[\bigtimes_{j \in J_1} \{f_j\} \times \bigtimes_{k \in J_2} \{g_k\} \times \bigtimes_{l \in J_3} \{h_l\} \bigg]. 
\end{align}
For any $l \in J$,
\begin{align}
Pr(y_l = i_l \mid A) &= \frac{Pr\bigg[ \bigtimes_{k \in \bar{J}} \{c_k\} \times \{i_l\} \times \bigtimes_{j \in J^{(-l)}} \bar{H}_j \bigg]}{Pr(A)} \nonumber \\
& = \frac{\pi_{\bfi}}{Pr(A)} \sum_{\bfa^{(l)} \in \barh_{J^{(-l)}}} \frac{\pi^{(\bar{J}, \{l\}, J^{(-l)})}_{c_k i_l \alpha^{(l)}_j} }{\pi_{\bfi}}. \label{eq:i_lgivA}
\end{align}
In the above display, we adopt the notation in \eqref{eq:pi_altnot}, with $V$ partitioned into $(\bar{J}, \{l\}, J^{(-l)})$ and 
$$
\pi^{(\bar{J}, \{l\}, J^{(-l)})}_{c_k i_l \alpha^{(l)}_j} = Pr \bigg[ \bigtimes_{k \in \bar{J} } \{c_k\} \times \{i_l\} \times \bigtimes_{j \in J^{(-l)}} \{ \alpha_j^{(l)}  \} \bigg].
$$
From \eqref{eq:i_lgivA}, we have
\begin{align*}
\prod_{l \in J} Pr(y_l = i_l \mid A) = \bigg[\frac{\pi_{\bfi}}{Pr(A)}\bigg]^{|J|} \sum_{ \bfa^{(q)} \in \barh_{J^{(-q)}} } \cdots \sum_{ \bfa^{(p)} \in \barh_{J^{(-p)}} }  \prod_{l \in J} \frac{\pi^{(\bar{J}, \{l\}, J^{(-l)})}_{c_k i_l \alpha^{(l)}_j} }{\pi_{\bfi}}.
\end{align*}
Substituting this in \eqref{eq:lvind3}, we have \eqref{eq:lvind3} is equivalent to showing
\begin{align}\label{eq:lvind4}
\bigg[\frac{Pr(A)}{\pi_{\bfi}}\bigg]^{|J| - 1} = \sum_{ \bfa^{(q)} \in \barh_{J^{(-q)}} } \cdots \sum_{ \bfa^{(p)} \in \barh_{J^{(-p)}} }  \prod_{l \in J} \frac{\pi^{(\bar{J}, \{l\}, J^{(-l)})}_{c_k i_l \alpha^{(l)}_j} }{\pi_{\bfi}}.
\end{align}
Recalling the set $A$ from \eqref{eq:A_def}, we have
$$
\frac{Pr(A)}{\pi_{\bfi}} = \sum_{\bfa \in \barh_J} \frac{ \pi^{(\bar{J}, J)}_{c_k \alpha_j} }{\pi_{\bfi}},
$$
implying
\begin{align}\label{eq:lvind5}
\bigg[\frac{Pr(A)}{\pi_{\bfi}}\bigg]^{|J| - 1} = \sum_{\bfa_q \in \barh_J} \cdots \sum_{\bfa_{p-1} \in \barh_J} \prod_{l \in J^{(-p)} } \frac{ \pi^{(\bar{J}, J)}_{c_k \alpha_{lj}} }{\pi_{\bfi}},
\end{align}
where $\bfa_q, \ldots, \bfa_{p-1}$ denote $|J| - 1$ independent copies of the running index $\bfa$, and $\alpha_{lj}$ is the entry in $\bfa_l$ corresponding to variable $j$. 

It now amounts to show that the expressions in the right hand side of \eqref{eq:lvind4} and \eqref{eq:lvind5} are the same. We first argue that both expressions contain the same number of terms. To see this, let $| \bar{H}_j| = m_j$. The expression of $Pr(y_l  = i_l \mid A)$ in \eqref{eq:i_lgivA} is a sum over $\prod_{j \neq l} m_j$ terms, and so $\prod_{l \in J} Pr(y_l  = i_l \mid A)$ has $\prod_{l \in J} \prod_{j \neq l} m_j = \prod_{l \in J} m_l^{(|J| - 1)}$ terms. Accordingly, the right hand side in \eqref{eq:lvind4} has $\prod_{l \in J} m_l^{(|J| - 1)}$ many terms. On the other hand, $Pr(A)/\pi_{\bfi}$ is a sum over $\prod_{j \in J} m_j$ terms, and hence $\{ Pr(A)/\pi_{\bfi} \}^{(|J| - 1)}$ in \eqref{eq:lvind5} also has $\prod_{j \in J} m_j^{(|J| - 1)}$ terms. 

Therefore, it now amounts to show that each term inside the summation in the right hand side of \eqref{eq:lvind4} has a one-to-one correspondence with a term in the right hand side of \eqref{eq:lvind5}. We establish this by showing
\begin{align}\label{eq:rat_red}
\prod_{l \in J} \frac{\pi^{(\bar{J}, \{l\}, J^{(-l)})}_{c_k i_l \alpha^{(l)}_j} }{\pi_{\bfi}} = \prod_{l \in J^{(-p)} } \frac{ \pi^{(\bar{J}, J)}_{c_k \alpha_{lj}} }{\pi_{\bfi}},
\end{align}
when for each $l$, $\alpha^{(l)}_j = \alpha_{lj}$ for all $j \neq l$. Introducing additional notation, let $\m E = \{E = E_1 \cup \{j\} : E_1 \subset \bar{J}, j \in J_2\}, \m E^{(-l)} = \{E = E_1 \cup \{j\} : E_1 \subset \bar{J}, j \in J^{(-l)} \}$ and $\m E^{(l)} = \{E = E_1 \cup \{l\} : E_1 \subset \bar{J} \}$. For any $l$, clearly $\m E$ is a disjoint union of $\m E^{(-l)}$ and $\m E^{(l)}$. Let $\bfi^{(l)}$ denote the cell such that $i^{(l)}_{k} = c_k$ for $k \in \bar{J}$ and $i^{(l)}_{j} = \alpha_{lj}$ for $j \in J$. 

First, consider the expression in the right hand side of \eqref{eq:rat_red}. We have 
\begin{align}
\frac{  \pi^{(\bar{J}, J)}_{c_k \alpha_{lj}} }{\pi_{\bfi}}
& = \exp \bigg[ \sum_{E \subset V} \bigg \{ \theta_E(\bfi^{(l)}_E) - \teie  \bigg\} \bigg] \nonumber \\
& = \exp \bigg[ \sum_{E \subset \m E} \bigg \{ \theta_E(\bfi^{(l)}_E) - \teie  \bigg\} \bigg] \nonumber \\
& = \exp \bigg[ \sum_{E \subset \m E^{(-l)}} \bigg \{ \theta_E(\bfi^{(l)}_E) - \teie  \bigg\} \bigg] \, 
       \exp \bigg[ \sum_{E \subset \m E^{(l)}} \bigg \{ \theta_E(\bfi^{(l)}_E) - \teie  \bigg\} \bigg] \label{eq:rat_rat_rhs}
\end{align}
The first equality in the above display simply follows from the expression of the cell probabilities for log-linear models in \eqref{eq:loglineargen}. The second inequality is the key one which uses (i) since $i^{(l)}_k = i_k = c_k$ for all $k \in \bar{J}$, all interaction terms corresponding to $E \subset \bar{J}$ cancel out between the numerator and denominator; and (ii) any $E \subset V$ such that $|E \cap J| \geq 2$, $\theta_E(\bfi^{(l)}_E) = \teie = 0$, given weak hierarchically and the condition $C_{\theta} = T_{C_{\theta}, H}$. To see this, suppose that there exists $E \subset V$ with $|E \cap J| \ge 2$ such that $\teie \ne 0$ for some $\bs{i} \in A$. By weak hierarchicality, there must be $j, j^* \in J$ such that $\theta_{\{j,j^*\}}(\alpha_j,\alpha_{j^*}) \ne 0$ for some $(\alpha_j, \alpha_{j^*}) \in \barh_j \times \barh_{j^*}$. Then $\theta_{\{j,j^*\}}(\alpha_j,\alpha_{j^*}) \notin T_{C_{\theta},H}$, contradicting $C_{\theta} = T_{C_{\theta}, H}$.

Using the same argument and additionally the fact that $\alpha^{(l)}_j = \alpha_{lj}$ for all $j \neq l$, we can simplify the expression in left hand side of \eqref{eq:rat_red} as
\begin{align}
\frac{\pi^{(\bar{J}, \{l\}, J^{(-l)})}_{c_k i_l \alpha^{(l)}_j} }{\pi_{\bfi}} = \exp \bigg[ \sum_{E \subset \m E^{(-l)}} \bigg \{ \theta_E(\bfi^{(l)}_E) - \teie  \bigg\} \bigg]. \label{eq:rat_red_lhs}
\end{align}
Therefore,
\begin{align}
& \prod_{l \in J^{(-p)} } \frac{ \pi^{(\bar{J}, J)}_{c_k \alpha_{lj}} }{\pi_{\bfi}} \nonumber \\
& = \prod_{l \in J^{(-p)} } \exp \bigg[ \sum_{E \subset \m E^{(-l)}} \bigg \{ \theta_E(\bfi^{(l)}_E) - \teie  \bigg\} \bigg] \, \prod_{l \in J^{(-p)} } \exp \bigg[ \sum_{E \subset \m E^{(l)}} \bigg \{ \theta_E(\bfi^{(l)}_E) - \teie  \bigg\} \bigg]  \nonumber \\
& = \prod_{l \in J} \exp \bigg[ \sum_{E \subset \m E^{(-l)}} \bigg \{ \theta_E(\bfi^{(l)}_E) - \teie  \bigg\} \bigg] = \prod_{l \in J} \frac{\pi^{(\bar{J}, \{l\}, J^{(-l)})}_{c_k i_l \alpha^{(l)}_j} }{\pi_{\bfi}} \nonumber,
\end{align}
establishing \eqref{eq:rat_red}. The second inequality in the above display used
\begin{align*}
\prod_{l \in J^{(-p)} } \exp \bigg[ \sum_{E \subset \m E^{(l)}} \bigg \{ \theta_E(\bfi^{(l)}_E) - \teie  \bigg\} \bigg] = \exp \bigg[ \sum_{E \subset \m E^{(-p)}} \bigg \{ \theta_E(\bfi^{(l)}_E) - \teie  \bigg\} \bigg],
\end{align*}
since $\m E^{(-p)} = \bigcup_{l \neq p} \m E^{(l)}$ is a disjoint union. 

\subsection*{Proof of (13) in Theorem \ref{thm:rankgeneraltight}}

The main idea in this part of the proof is that we can merge certain sets in $\m P_H^0$ to create a coarser partition without sacrificing the conditional independence. 

For a set $A = \bigtimes_{j \in V} A_j$ in $\m P_H^0$ and $J \subset V$, let $\Pi_J(A)$ denote 
$$\Pi_J(A) = \prod_{j \in J} A_j.$$
With a slight abuse of notation, we shall use $\Pi_l(A)$ to denote the $l$th coordinate projection, i.e., $\Pi_l(A) = A_l$. 

Fix $l \in V$ and let $V^{(-l)} = V \setminus \{l\}$. In this proof, we shall use $\bfa$ to denote a $V^{(-l)}$-cell suppressing the dependence on $l$. Given $\bfa$, let
\begin{align}\label{eq:AH_l}
\m P_{H, l}^{\bfa} = \big \{ A \in \m P_H^0 : \Pi_{V^{(-l)}}(A) = \bigtimes_{j \neq l} \{\alpha_j\}  \big \}.
\end{align} 
Let $\m A$ denote the collection of all $V^{(-l)}$-cells $\bfa$ such that $\m P_{H, l}^{\bfa}$ is non-empty. For $\bfa \in \m A$, let 
\begin{align}\label{eq:B_a}
B^{\bfa} = \bigcup_{A \in \m P_{H, l}^{\bfa}} A. 
\end{align}
Note that for any $\bfa \in \m A$, $\mid \m P_{H, l}^{\bfa} \mid = |H_l| + 1$, since $\Pi_l(A)$ ranges over the elements of $\m P_{H, l}$, i.e., $\{ i_l \}$ for $i_l \in H_l$ and $\bar{H}_l$. It is also evident that $B^{\bfa} = \bigtimes_{j \neq l} \{\alpha_j\} \times \m I_l$. 

We now create a coarser partition $\m P_H^{(l)}$ out of $\m P_H^0$ by replacing the collection of sets $\m P_{H, l}^{\bfa}$ by the single set $B^{\bfa}$ for every $\bfa \in \m A$, so that
\begin{align}\label{eq:PH_l}
\m P_{H, l} = \bigcup_{\bfa \in \m A} \bigg[ \big(\m P_H^0 \setminus \m P_{H, l}^{\bfa}  \big) \cup \big\{ B^{\bfa} \big\}  \bigg]. 
\end{align}
The main idea is that if $(|V| -1)$ coordinate projections $\Pi_j(A)$ are singletons $\{ \alpha_j \}$, we can simply set the $l$th coordinate projection of $A$ to be $\m I_l$ and achieve conditional independence \eqref{eq:lvind2}. This follows immediately from the expression in the display after \eqref{eq:part}. However, our construction of $\m P_H^0$ clearly contains sets of the form $\bigtimes_{j \neq l} \{\alpha_j\} \times \{i_l\}$ for $i_l \in H_l$ and $\bigtimes_{j \neq l} \{\alpha_j\}\times \bar{H}_l$ which are redundant. To avoid this redundancy, we merge these sets in $\m P_{H, l}^{\bfa}$ to form $B^{\bfa} = \bigtimes_{j \neq l} \{\alpha_j\} \times \m I_l$ for every $\bfa \in \m A$. 

It only remains to calculate the cardinality of $\m P_{H, l}$ now. As pointed out in the previous paragraph, $\mid \m P_{H, l}^{\bfa} \mid = |H_l | + 1$ for all $\bfa \in \m A$, and hence the net reduction in the number of elements from $\m P_H^0$ to $\m P_{H, l}$ is 
$$
\m P_H^0  - \m P_{H, l} =  \mid \m A \mid \, \mid H_l \mid. 
$$
It thus remains to calculate $|\m A|$. We need to count the number of distinct $\bfa$ such that \eqref{eq:AH_l} is satisfied. Recall that for any $A \in \m P_H^0$ and any $j \in V$, $\Pi_j(A)$ ranges over the elements of the partition $\m P_{H, j}$. The number of singleton sets in $\m P_{H, j}$ is $| H_j|$ as long as $|H_j| < (d-1)$ (the sets $\{i_j \}$ for $i_j \in H_j$). However, when $|H_j| = (d-1)$, $\bar{H}_j$ is also a singleton set and hence the number of singleton sets in $\m P_{H, j}$ in that case becomes $|H_j| + 1$. Therefore, we conclude, 
$$
\m A = \bigg[ \prod_{j \neq l: |H_j| = d-1} (\mid H_j \mid + 1) \bigg] \bigg[ \prod_{j \neq l : |H_j| < d-1 } \mid H_j \mid \bigg]. 
$$
The proof is completed by noting $| \m A| |H_l | = \prod_{j \in W_l} (|H_j| +1) \prod_{j \in \bar{W}_l} |H_j|$ and taking minimum over $l \in V$ and $H \in \ms{H}$.

\subsection{Proof of Theorem \ref{thm:indparafac}}
The condition that $\bigwedge_{F \in \ms{F}_n} |F| \asymp \log_2(p_n)$ gives that for each clique $F$ the number of nonzero parameters corresponding to that clique grows linearly in $p_n$. This follows because there are at most $2^{p_0}$ parameters for any clique $F_0$ with $|F_0| = p_0$, so we may have up to $2^{\lceil \log_2(p_n) \rceil} \approx p_n$ parameters corresponding to each clique. There are $k_n$ cliques, each of which corresponds to a marginal probability tensor $\pi_n^{(l)}$ with $\rnkp(\pi_n^{(l)}) \asymp \mc{O}(p_n)$. So the joint distribution can be represented by the Hadamard product of $k_n$ probability tensors $\pi_n^{(1)},\ldots,\pi_n^{(k_n)}$, with $\rnkp(\pi_n^{(l)}) \asymp p_n$ for every $l = 1,\ldots,k_n$. Thus, $\sum_{s=1}^{k_n} \rnkp(\pi_n^{(s)}) \asymp k_n p_n$.

\begin{supplement}
\sname{Supplement A}\label{suppA}
\stitle{Title of the Supplement A}
\slink[url]{http://www.e-publications.org/ims/support/dowload/imsart-ims.zip}
\sdescription{Dum esset rex in
accubitu suo, nardus mea dedit odorem suavitatis. Quoniam confortavit
seras portarum tuarum, benedixit filiis tuis in te. Qui posuit fines tuos}

\subsection{Supplemental results}
The following proposition shows that in the two-dimensional case, the nonnegative rank can be bounded by one plus the minimum number of rows and columns that contain all of the cells that differ from a rank one nonnegative matrix. Figure \ref{fig:ranks2d} shows several examples of the essential principle the proof, which is constructive. Although in the case of probability tensors corresponding to log-linear models, this result is a corollary of Theorem \ref{thm:rankgeneraltight}, the constructive approach is very instructive and provided intuition for the general result.

\begin{proposition} \label{prop:rank2d}
Suppose $M$ is a $d \times d$ nonnegative matrix. Let $\lambda^{(1)}, \lambda^{(2)}$ be nonnegative vectors and set $\tilde{M} = \lambda^{(1)} \otimes \lambda^{(2)}$ with
\begin{align*}
&C_M = \{(c_1,c_2) : M_{c_1 c_2} - \tilde{M}_{c_1 c_2} \ne 0\}, &C_M^{(1)} = \{c_1 : (c_1,c_2) \in C_M \} \\
&C_M^{(2)} = \{c_2 : (c_1,c_2) \in C_M\}, 
\end{align*}
and $\ms{H} = \{ H : T_{(C_M,H)} = C_M \}$. Define $|H| = |H_1| + |H_2|$. Then $\rnkp(M) \le 1+\bigwedge_{H \in \ms{H}} |H|$.
\end{proposition}

\begin{proof}
Let $H = (H_1,H_2)$ be any element of $\ms{H}$. Set
\begin{align*}
&\lambda_{0 c_1}^{(1)} = \lambda^{(1)} \ind(c_1 \notin H_1), \text{ and } \\ 
&\lambda_{0 c_2}^{(2)} = \lambda^{(2)} \ind(c_2 \notin H_2), 
\end{align*}
and put $M^{(0)} = \lambda_0^{(1)} \otimes \lambda_0^{(2)}$. Then for $1\le h \le \mid H_1 \mid$ set
\begin{align*}
&\lambda^{(1)}_{h c_1} = \ind(c_1 = H_{1h}), \text{ and } \\
&\lambda^{(2)}_{h c_2} = M_{H_{1h} c_2},
\end{align*}
where $H_{1h}$ is the $h$th element of (any ordering of) $H_1$. Then set $M^{(1)} = \sum_{h=1}^{\mid H_1 \mid} \lambda^{(1)}_{h} \otimes \lambda^{(2)}_{h}$. Finally for $1 \le h \le |H_2|$ set
\begin{align*}
&\lambda^{(1)}_{h c_1} = M_{c_1 H_{2h}} \ind(c_1 \notin H_1) \text{ and, } \\
&\lambda^{(2)}_{h c_2} = \ind(c_2 = H_{2h}),
\end{align*}
and put $M^{(2)} =  \sum_{h=1}^{|H_1|}\lambda^{(1)}_{h} \otimes \lambda^{(2)}_{h}$. Then $M^{(0)} + M^{(1)} + M^{(2)} = M$ and therefore $M$ has a $1 + |H| =  1+\bigwedge_{H' \in \ms{H}} (|H'|)$-term nonnegative PARAFAC expansion. 
\end{proof}

\begin{figure}[h]
 \includegraphics[width=\textwidth]{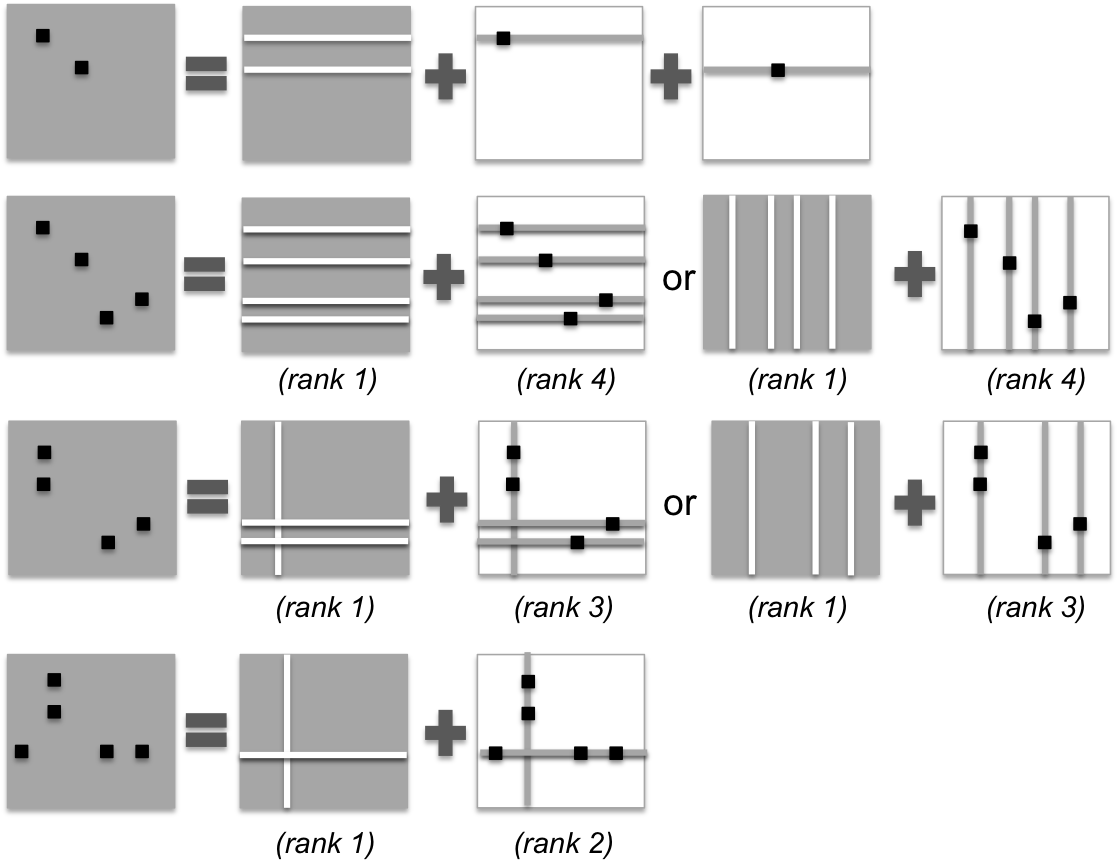}
 \caption{examples} \label{fig:ranks2d}
\end{figure}

\subsection{Supplemental examples}
We provide a more complicated example of the principles underlying Theorem \ref{thm:rankgeneraltight} that includes nonzero three-way interactions. Let $p=5$ with $d$ arbitrary and suppose $S_{\theta}$ is given by
\begin{align*}
 &\theta^{(12)}_{2 c_2} \ne 0 \text{ for } c_2 \ge 2 &\theta^{(23)}_{2 c_3} \ne 0 \text{ for } c_3 \ge 2 \\
 &\theta^{(34)}_{2 c_4} \ne 0 \text{ for } c_4 \ge 2 &\theta^{(45)}_{2 c_5} \ne 0 \text{ for } c_5 \ge 2 \\
 &\theta^{(15)}_{c_1 2} \ne 0 \text{ for } c_1 \ge 2 &\theta^{(24)}_{2 c_4} \ne 0 \text{ for } c_4 \ge 2 \\
 &\theta^{(14)}_{2 c_4} \ne 0 \text{ for } c_4 \ge 2 &\theta^{(124)}_{224} \ne 0 \\
 &\theta^{(25)}_{2 c_5} \ne 0 \text{ for } c_5 \ge 2 &\theta^{(15)}_{2 c_5} \ne 0 \text{ for } c_5 \ge 2 \\
 &\theta^{(125)}_{224} \ne 0,
\end{align*}
so there are two nonzero three-way interactions. As in example \ref{ex:illustrate}, Theorem \ref{thm:rankgeneral} gives the trivial bound of $d^4$ for all $5!=120$ permutations. Letting $H = \{\{2\},\{2\},\{2\},\{2\},\{2\}\} \in \ms{H}$, we know $\rnkp(\pi) \le 2^5 = 32$. Consider the event $A = \{2\} \times \{2\} \times \{2\} \times \barh_4\times \barh_5 \in \mc{P}^0_H$
and the cell $\bs{i} = (2,2,2,4,4)$. We show that (\eqref{eq:lvind1}) holds with $A$ and $\bs{i}$, i.e. that
\begin{align*}
 Pr(A^* \mid A) &= Pr(y_1 = 2 \mid A) Pr(y_2 = 2 \mid A) Pr(y_3 = 2 \mid A) Pr(y_4 = 4 \mid A) Pr(y_5 = 4 \mid A) \\
 &= 1 \times 1 \times 1 \times Pr(y_4 = 4 \mid A) Pr(y_5 = 5 \mid A).
\end{align*}
Following the proof of Theorem \ref{thm:rankgeneraltight}, this is equivalent to showing
\begin{align*}
 \frac{Pr(A)}{\pi_{22244}} = \sum_{c_4 \ne 2} \sum_{c_5 \ne 2} &\frac{\pi_{2224c_5}}{\pi_{22244}} \frac{\pi_{222c_44}}{\pi_{22244}}.
\end{align*}
Since
\begin{align*}
\frac{Pr(A)}{\pi_{22244}} = \sum_{c_4 \ne 2} \sum_{c_5 \ne 2} &\frac{\pi_{222c_4 c_5}}{\pi_{22244}} = \sum_{c_4 \ne 2} \sum_{c_5 \ne 2} \frac{\pi_{222c_4 c_5}}{\pi_{222c_4 4}} \frac{\pi_{222c_4 4}}{\pi_{22244}},
\end{align*}
we need to show that
\begin{align*}
 \frac{\pi_{222c_4 c_5}}{\pi_{222c_4 4}} = \frac{\pi_{2224 c_5}}{\pi_{22244}}.
\end{align*}
All main effects and interactions that correspond to variables $y_1,\ldots,y_4$ will be eliminated in the ratios on both sides, so we focus only on those involving $y_5$. This gives us that the left side of the above display (assuming $c_5 \ne 4$) is
\begin{align*}
 \exp(\theta^{(5)}_{c_5} - \theta^{(5)}_4 + \theta^{(15)}_{2 c_5} - \theta^{(15)}_{24} + \theta^{(25)}_{2 c_5} - \theta^{(25)}_{24} - \theta^{(125)}_{224}).  
\end{align*}
The right side differs only in the value of $y_4$, but since there are no $4,5$ interactions at these levels of the variables and the level of $y_4$ is the same in the numerator and denominator on the right side, the right side has the same value as the left side, despite the fact that there are nonzero three-way interactions. Note that $\theta^{(124)}_{224}$ cancelled on the right side and was either zero or cancelled on the left side as well (the latter occuring when $c_4 = 4$).

We now utilize the same setup to demonstrate the key principle in the proof of \ref{eq:equality} in Theorem \ref{thm:rankgeneraltight}. This principle can be described succinctly as the failure of conditional independence upon replacing sets in the partition $\mc{A}_H^0$ with their union when these sets do not have in common at least $p-1$ singleton events. Let 
\begin{align*}
 A^{\gamma} &= \{\{2\},\{2\},\{2\},\{\ne 2\},\{\ne 2\}\} \\
 A^{\beta} &= \{\{2\},\{2\},\{2\},\{2\},\{\ne 2\}\} \\
 A^* &= \{\{2\},\{2\},\{2\},\{4\},\{4\}\}.
\end{align*}
and note that $A^{\gamma}$ and $A^{\beta}$ share $3=p-2$ singleton events. Then
\begin{align*}
 A^{\gamma} \cup A^{\beta} = \{\{2\},\{2\},\{2\},\mc{I}_4,\{\ne 2\}\}.
\end{align*}
Now we want to show that
\begin{align*}
 Pr(A^* \mid A) \ne Pr(\mc{I}_4 \mid A) Pr(\ne 2 \mid A).
\end{align*}
This will be true iff
\begin{align*}
 \frac{\pi_{2 2 2 c_4 c_5}}{\pi_{2 2 2 c_4 4}} \ne \frac{\pi_{2 2 2 4 c_5}}{\pi_{2 2 2 4 4}}
\end{align*}
for one or more values of $c_4 \in A_4, c_5 \in A_5$. Here, unlike our previous example using this setup, $c_4$ can take any value in $\mc{I}_4$, \emph{including the value 2}. However, $\theta_{\{4,5\}}(2,c_5) \ne 0$ for any $c_5 \ge 2$. So now on the LHS we get
\begin{align*}
 \exp\big\{&\theta_{\{5\}}(c_5) - \theta_{\{5\}}(4) + \theta_{\{1,5\}}(2,c_5) - \theta_{\{1,5\}}(2,4) + \theta_{\{2,5\}}(2,c_5) - \\
 &\theta_{\{2,5\}}(2,4) - \theta_{\{1,2,5\}}(2,2,4) + \theta_{\{4,5\}}(2,c_5)-\theta_{\{4,5\}}(2,4)\big\}.  
\end{align*}
when $c_4 = 2$ and $c_5 \ne 4$. But on the RHS we still get
\begin{align*}
 \exp\big\{&\theta_{\{5\}}(c_5) - \theta_{\{5\}}(4) + \theta_{\{1,5\}}(2,c_5) - \theta_{\{1,5\}}(2,4) + \\
 &\theta_{\{2,5\}}(2, c_5) - \theta_{\{2,5\}}(2,4) - \theta_{\{1,2,5\}}(2,2,4)\big\}
\end{align*}
always, so there are events contained in $A$ where the equality fails, and therefore conditional independence does not hold.

\subsection{Posterior computation for c-Tucker models} \label{app:computation}

The conditional posteriors for all the parameters can be derived in closed form using standard algebra and the sampler cycles through the following steps,
\begin{description}
\item[$Step \, 1.$] For $j : s_j = s$ and $h = 1, \ldots, m$, update $\bfl_h^{(j)}$ from the
following Dirichlet full conditional posterior distribution,
\begin{eqnarray*}
\pi(\bfl_h^{(j)} \mid -) \sim \mbox{Diri}\bigg(a_{j1} + \sum_{i:z_{is} = h} 1(y_{ij} = 1), \ldots, a_{jd_j} + \sum_{i:z_{is} = h} 1(y_{ij} = d_j)\bigg).
\end{eqnarray*}

\item[$Step \, 2.$] Sample the latent class indicators $z_{is}$ for $s \in \{1,\ldots,k\}$ from the following full conditional distribution,
\begin{align*}
\mbox{pr}(z_{is} = h_s \mid -) \propto \bigg(  \prod_{j: s_j = s} \lambda_{h_s y_{ij}}^{(j)} \bigg) \psi_{w_ih_s}^{(s)}, \, h_s = 1, \ldots, m. \notag
\end{align*}

\item[$Step \, 3.$] Sample $w_i$ from the following full-conditional distribution,
\begin{align*}
\mbox{pr}(w_i = l \mid -) \propto \nu_l ~ \prod_{s=1}^k \psi_{l z_{is}}^{(s)}, \, l = 1, \ldots, k.
\end{align*}

\item[$Step \, 4.$] Sample $\nu_l^*$ from the following full-conditional distribution,
\begin{align*}
\pi(\nu_l^* \mid -) \sim \mbox{beta}(1+m_l, \beta + m_{l+}) \, l = 1, \ldots, k,
\end{align*}
where $m_l = \sum_{i=1}^n 1(w_i = l)$ and $m_{l+} = \sum_{i=1}^n 1(w_i > l)$.

\item[$Step \, 5.$] To update $\phi_{lh}^{(s)}$ for $s \in \{1,\ldots,k\}$ define $n_{lh}^{(s)} = \sum_{i:w_i=l} 1(z_{is} = h)$ and
$n_{lh+}^{(s)} = \sum_{i: w_i = l} 1(z_{is} > h)$. Then, the full conditional posterior of $\phi_{lh}^{(s)}$ is
\begin{align*}
\pi(\phi_{lh}^{(s)} \mid -) \sim \mbox{Beta}\bigg(1 + n_{lh}^{(s)}, \delta_1 + n_{lh+}^{(s)}\bigg).
\end{align*}

\item[$Step \, 6.$] Assuming a gamma$(a_{\beta}, b_{\beta})$ prior for $\beta$, the full conditional posterior is
\begin{align*}
\pi(\beta \mid -) \sim \mbox{gamma}\bigg(a_{\beta} + k, b_{\beta} - \sum_{l = 1}^k \log(1 - \nu_l^*)\bigg).
\end{align*}

\item[$Step \, 7.$] Assuming a gamma$(a_{\delta}^{(s)}, b_{\delta}^{(s)})$ prior for $\delta_s$ for each $s \in \{1,\ldots,k\}$, the full conditional posterior is
\begin{align*}
\pi(\delta_1 \mid -) \sim \mbox{gamma}\bigg(a_{\delta}^{(s)} + mk, b_{\delta}^{(s)} - \sum_{l = 1}^k \sum_{h = 1}^m \log(1 - \phi_{lh}^{(s)})\bigg).
\end{align*}

\item[$Step \, 8.$] The groups $s_j$ are updated sequentially. Set $L_{ijl} = Pr(y_j = y_{ij} | s_j = l) = \lambda^{(j)}_{z_{l} y_j}$, and $L_{jl} = \prod_i L_{ijl}$. Then sample $s_j$ as
\begin{align*}
\pi(s_j = l \mid -) = \frac{\xi_l L_{jl}}{\sum_{l=1}^k \xi_l L_{jl}}.
\end{align*}

\item[$Step \, 9.$] Let $n_l = \sum_j \ind_{s_j = l}$, and sample $\xi$ from
\begin{align*}
p(\xi \mid -) \sim \mbox{Dirichlet}(n_1 + 1/k,\ldots,n_k + 1/k).
\end{align*}

\end{description}

\subsection{Supplemental figures for section 6}

\begin{landscape}
\begin{figure}[ht]
\includegraphics[width=8in]{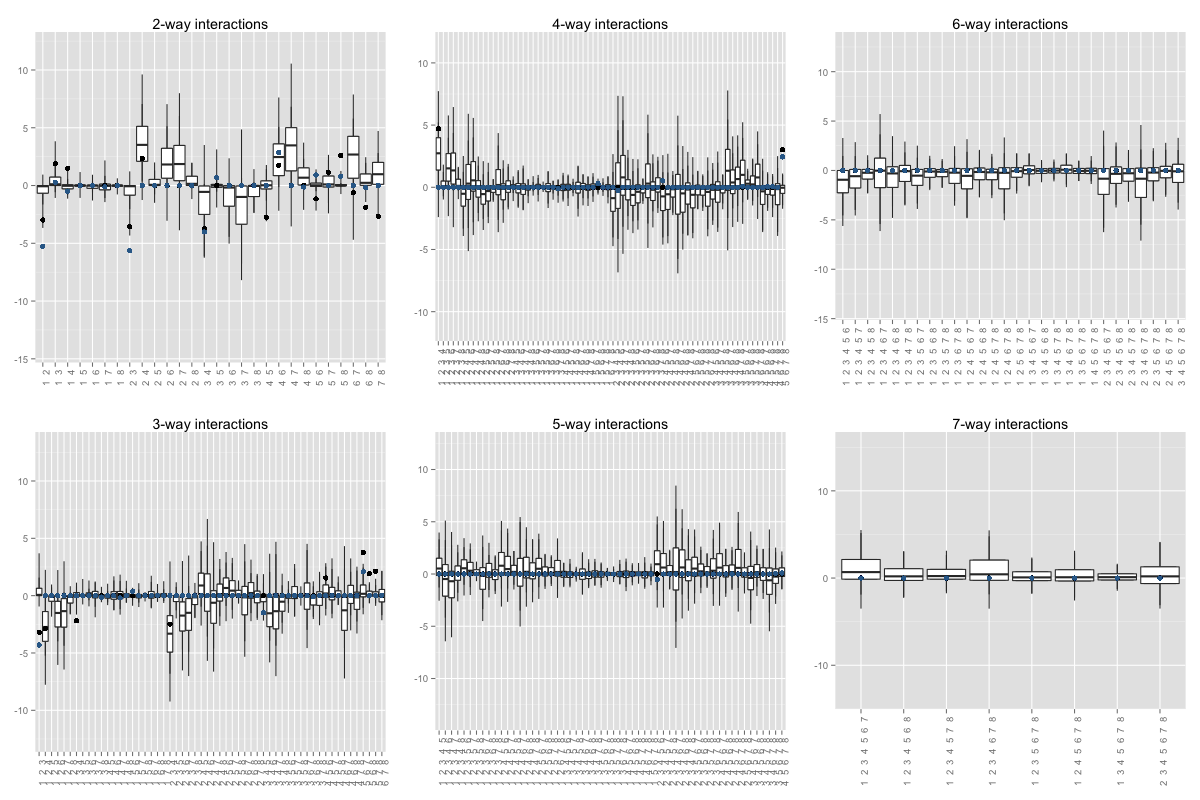} \\
\caption{Boxplots of posterior samples for interaction terms for simulation example where variable groups are learned. Black markers indicate the true parameter values in the simulation, blue markers indicate lasso parameter estimates. } \label{fig:learngroupsbox}
\end{figure}

\begin{table}[h]
\begin{footnotesize}
\begin{tabular}{cccccccccccccccccc}
  \hline
 & \multicolumn{7}{c}{\bf{ADL}} & \multicolumn{10}{c}{\bf{IADL}} \\\cline{2-7} \cline{9-18}
 & \bf{1} & \bf{2} & \bf{3} & \bf{4} & \bf{5} & \bf{6} & & \bf{1} & \bf{2} & \bf{3} & \bf{4} & \bf{5} & \bf{6} & \bf{7} & \bf{8} & \bf{9} & \bf{10} \\ 
  \hline
ADL & \multicolumn{16}{c}{} \\
1 & - & 1.00 & 1.00 & 0.00 & 0.00 & 0.37 & & 1.00 & 1.00 & 1.00 & 0.00 & 0.00 & 0.00 & 1.00 & 0.00 & 0.00 & 1.00 \\ 
  2 & 0.21 & - & 1.00 & 0.00 & 0.00 & 1.00 & & 1.00 & 1.00 & 1.00 & 0.00 & 0.99 & 0.99 & 1.00 & 0.01 & 1.00 & 1.00 \\ 
  3 & 0.26 & 0.25 & - & 1.00 & 0.00 & 1.00 & & 1.00 & 1.00 & 1.00 & 0.00 & 0.25 & 0.17 & 1.00 & 0.00 & 0.68 & 1.00 \\ 
  4 & 0.06 & 0.09 & 0.12 & - & 1.00 & 1.00 & & 1.00 & 1.00 & 1.00 & 1.00 & 0.05 & 1.00 & 0.06 & 1.00 & 0.99 & 0.00 \\ 
  5 & 0.03 & 0.06 & 0.05 & 0.20 & - & 1.00 & & 0.99 & 0.99 & 0.97 & 1.00 & 0.98 & 1.00 & 0.36 & 1.00 & 1.00 & 0.00 \\ 
  6 & 0.10 & 0.14 & 0.18 & 0.38 & 0.21 & - & & 1.00 & 1.00 & 1.00 & 1.00 & 0.99 & 1.00 & 1.00 & 1.00 & 1.00 & 0.00 \\ 
IADL &  \multicolumn{16}{c}{} \\
  1 & 0.20 & 0.27 & 0.25 & 0.16 & 0.11 & 0.27 & & - & 1.00 & 1.00 & 1.00 & 1.00 & 1.00 & 1.00 & 1.00 & 1.00 & 1.00 \\ 
  2 & 0.14 & 0.20 & 0.19 & 0.20 & 0.15 & 0.32 & & 0.42 & - & 1.00 & 1.00 & 0.99 & 1.00 & 1.00 & 1.00 & 1.00 & 1.00 \\ 
  3 & 0.18 & 0.24 & 0.20 & 0.13 & 0.11 & 0.22 & & 0.48 & 0.37 & - & 1.00 & 1.00 & 1.00 & 1.00 & 1.00 & 1.00 & 1.00 \\ 
  4 & 0.04 & 0.08 & 0.08 & 0.19 & 0.19 & 0.25 & & 0.14 & 0.21 & 0.12 & - & 0.04 & 1.00 & 0.00 & 1.00 & 0.99 & 0.00 \\ 
  5 & 0.08 & 0.14 & 0.10 & 0.09 & 0.13 & 0.14 & & 0.20 & 0.17 & 0.22 & 0.09 & - & 1.00 & 1.00 & 1.00 & 1.00 & 0.99 \\ 
  6 & 0.07 & 0.12 & 0.10 & 0.13 & 0.18 & 0.19 & & 0.19 & 0.21 & 0.19 & 0.15 & 0.22 & - & 1.00 & 1.00 & 1.00 & 1.00 \\ 
  7 & 0.15 & 0.22 & 0.15 & 0.09 & 0.10 & 0.15 & & 0.30 & 0.22 & 0.32 & 0.09 & 0.31 & 0.21 & - & 1.00 & 1.00 & 1.00 \\ 
  8 & 0.05 & 0.09 & 0.07 & 0.13 & 0.37 & 0.17 & & 0.16 & 0.18 & 0.18 & 0.14 & 0.26 & 0.25 & 0.21 & - & 1.00 & 0.99 \\ 
  9 & 0.09 & 0.14 & 0.10 & 0.11 & 0.19 & 0.16 & & 0.22 & 0.21 & 0.24 & 0.11 & 0.30 & 0.24 & 0.31 & 0.41 & - & 1.00 \\ 
  10 & 0.17 & 0.19 & 0.15 & 0.06 & 0.05 & 0.09 & & 0.22 & 0.15 & 0.24 & 0.05 & 0.21 & 0.13 & 0.33 & 0.12 & 0.21 & - \\ 
   \hline
\end{tabular}
\end{footnotesize}
\caption{ Estimated Cram\'{e}r's V associations (elements under the main diagonal) and posterior probabilities $Pr(H_{1,\rho}|y^{(1:n)})$ (elements above the main diagonal) in the NLTCS data estimated using the c-Tucker model.} \label{tab:pwcv_nltcs}
\end{table}

\begin{table}[h]
\begin{footnotesize}
\begin{tabular}{cccccccccccccccccc}
  \hline
 & \multicolumn{7}{c}{\bf{ADL}} & \multicolumn{10}{c}{\bf{IADL}} \\\cline{2-7} \cline{9-18}
 & \bf{1} & \bf{2} & \bf{3} & \bf{4} & \bf{5} & \bf{6} & & \bf{1} & \bf{2} & \bf{3} & \bf{4} & \bf{5} & \bf{6} & \bf{7} & \bf{8} & \bf{9} & \bf{10} \\ 
  \hline
ADL & \multicolumn{16}{c}{} \\
1 & - & 1.00 & 1.00 & 0.00 & 0.00 & 0.61 & & 1.00 & 1.00 & 1.00 & 0.00 & 0.00 & 0.00 & 0.99 & 0.00 & 0.00 & 1.00 \\ 
  2 & 0.21 & - & 1.00 & 0.43 & 0.00 & 1.00 & & 1.00 & 1.00 & 1.00 & 0.00 & 0.99 & 1.00 & 1.00 & 0.08 & 1.00 & 1.00 \\ 
  3 & 0.26 & 0.25 & - & 1.00 & 0.00 & 1.00 & & 1.00 & 1.00 & 1.00 & 0.00 & 0.05 & 0.45 & 1.00 & 0.00 & 0.14 & 0.99 \\ 
  4 & 0.07 & 0.10 & 0.15 & - & 1.00 & 1.00 & & 1.00 & 1.00 & 1.00 & 1.00 & 0.38 & 1.00 & 0.32 & 1.00 & 0.98 & 0.00 \\ 
  5 & 0.03 & 0.06 & 0.05 & 0.21 & - & 1.00 & & 0.99 & 1.00 & 0.98 & 1.00 & 1.00 & 1.00 & 0.32 & 1.00 & 1.00 & 0.00 \\ 
  6 & 0.10 & 0.14 & 0.20 & 0.38 & 0.21 & - & & 1.00 & 1.00 & 1.00 & 1.00 & 1.00 & 1.00 & 1.00 & 1.00 & 1.00 & 0.03 \\ 
IADL &  \multicolumn{16}{c}{} \\
  1 & 0.21 & 0.28 & 0.28 & 0.18 & 0.11 & 0.28 & & - & 1.00 & 1.00 & 1.00 & 1.00 & 1.00 & 1.00 & 1.00 & 1.00 & 1.00 \\ 
  2 & 0.14 & 0.19 & 0.19 & 0.21 & 0.16 & 0.34 & & 0.43 & - & 1.00 & 1.00 & 1.00 & 1.00 & 1.00 & 1.00 & 1.00 & 1.00 \\ 
  3 & 0.16 & 0.22 & 0.18 & 0.13 & 0.11 & 0.22 & & 0.48 & 0.40 & - & 1.00 & 1.00 & 1.00 & 1.00 & 1.00 & 1.00 & 1.00 \\ 
  4 & 0.04 & 0.08 & 0.08 & 0.19 & 0.18 & 0.25 & & 0.15 & 0.23 & 0.13 & - & 0.33 & 1.00 & 0.13 & 1.00 & 1.00 & 0.00 \\ 
  5 & 0.06 & 0.12 & 0.09 & 0.10 & 0.14 & 0.14 & & 0.18 & 0.17 & 0.19 & 0.10 & - & 1.00 & 1.00 & 1.00 & 1.00 & 1.00 \\ 
  6 & 0.06 & 0.12 & 0.10 & 0.14 & 0.19 & 0.20 & & 0.19 & 0.21 & 0.18 & 0.17 & 0.27 & - & 1.00 & 1.00 & 1.00 & 1.00 \\ 
  7 & 0.13 & 0.21 & 0.14 & 0.10 & 0.10 & 0.14 & & 0.27 & 0.21 & 0.28 & 0.09 & 0.30 & 0.23 & - & 1.00 & 1.00 & 1.00 \\ 
  8 & 0.05 & 0.09 & 0.07 & 0.14 & 0.39 & 0.19 & & 0.16 & 0.19 & 0.17 & 0.15 & 0.26 & 0.26 & 0.20 & - & 1.00 & 0.95 \\ 
  9 & 0.07 & 0.13 & 0.09 & 0.11 & 0.20 & 0.16 & & 0.20 & 0.21 & 0.23 & 0.12 & 0.31 & 0.25 & 0.30 & 0.42 & - & 1.00 \\ 
  10 & 0.14 & 0.16 & 0.13 & 0.06 & 0.05 & 0.09 & & 0.20 & 0.13 & 0.20 & 0.05 & 0.19 & 0.12 & 0.32 & 0.11 & 0.20 & -   \\ 
   \hline
\end{tabular}
 \end{footnotesize}
 \caption{Estimated Cram\'{e}r's V associations (elements under the main diagonal) and posterior probabilities $Pr(H_{1,\rho}|y^{(1:n)})$ (elements above the main diagonal) in the NLTCS data estimated using copula Gaussian graphical model in \cite{dobra2011copula}.} \label{tab:pwcv_nltcs_copula}
\end{table}
\end{landscape}
\end{supplement}


\bibliographystyle{plainnat}
\bibliography{bptd_bib.bib}
\end{document}